\documentclass[a4paper]{amsart}
\usepackage{amsfonts}
\usepackage{amsmath}
\usepackage{amssymb}

\theoremstyle{plain}
\newtheorem{theorem}{Theorem}[section]

\newtheorem{corollary}{Corollary}[section]

\newtheorem{lemma}{Lemma}[section]

\numberwithin{equation}{section}

\input{tcilatex}

\begin{document}
\title[Discrete extrapolation theorem]{Discrete Rubio de Francia
extrapolation theorem via factorization of weights and iterated algorithms}
\author{\textbf{S. H. Saker}$^{1}$,\textbf{\ A. I. Saied}$^{2}$ \textbf{and
R. P. Agarwal}$^{3}$}
\address{$^{1}$Department of Mathematics, Faculty of Science, Mansoura
University, Mansoura-Egypt, e-mail: shsaker@mans.edu.eg\\
$^{2}$Department of Mathematics, Faculty of Science,\ Benha university,
Benha- Egypt, e-mail: as0863289@gmail.com\\
$^{3}$Department of Mathematics, Texas A \& M University- Kingsvilie, \\
Texas, 78363, USA.\\
E-mail: Ravi.Agarwal@tamuk.edu}
\maketitle

\begin{abstract}
In this paper, we prove a discrete Rubio de Francia extrapolation theorem
via factorization of discrete Muckenhoupt weights and discrete iterated
Rubio de Francia algorithm and its duality. Precisely, we will prove that if
for some $p_{0}>1$ and $w\in \mathcal{A}_{p_{0}}$ the $\ell ^{p_{0}}(w)$
inequality $\left\Vert \mathcal{T}f\right\Vert _{\ell ^{p_{0}}(w)}\leq
C\left\Vert f\right\Vert _{\ell ^{p_{0}}(w)}$ holds for pairs of the form ($%
\mathcal{T}f,f)$, then for all $p$ and $w\in \mathcal{A}_{p}$ the $\ell
^{p}(w)$ inequality $\left\Vert \mathcal{T}f\right\Vert _{\ell ^{p}(w)}\leq
C\left\Vert f\right\Vert _{\ell ^{p}(w)}$ also holds for such pairs, where $%
\mathcal{T}$ is a given operator. The new approach, among other things,
allows us to obtain explicit estimates of the constants and do not assume
that the operator $\mathcal{T}$ is linear or even sublinear. The theorem
will be proved for a fixed\ exponent $1<p<\infty $ and the case when $0<p<1$
is an open problem.

\bigskip


\bigskip

\end{abstract}

\section{Introduction}

The study of boundedness of operators is one of the main subjects in modern
harmonic analysis. The ideas related to the boundedness of operators and the
weighted inequalities appeared with the birth of singular integrals but a
better understanding of the subject was only in the 1970s. In 1972
Muckenhoupt \cite{M} introduced the $A_{p}-$weights (now often referred to
as Muckenhoupt weights): for a given exponent $p>1$ a weight $w$ is said to
belongs to $A_{p}$ if there exists a real constant $\mathcal{C}>1$ such
that, 
\begin{equation}
\left( \frac{1}{\left\vert I\right\vert }\int_{I}w(x)dx\right) \left( \frac{1%
}{\left\vert I\right\vert }\int_{I}w^{-\frac{1}{p-1}}(x)dx\right) ^{p-1}\leq 
\mathcal{C},  \label{Ap}
\end{equation}%
holds for every interval $I\subseteq I_{0}$. The infimum of all such $%
\mathcal{C}$ is denoted by $[w]_{A_{p}}$ and is called the $A_{p}\emph{-}$%
norm. The full characterization of the weights $w,$ for which the
Hardy-Littlewood maximal operator%
\begin{equation}
\mathcal{M}f(x):=\sup_{x\in I}\frac{1}{\left\vert I\right\vert }%
\int_{I}f(y)dy,  \label{hl}
\end{equation}%
is bounded on $L_{w}^{p}(\mathbb{R}^{+})$ by means of the so-called $A_{p}$%
\emph{-condition} was achieved by Muckenhoupt and the weighted norm
inequality%
\begin{equation}
\int_{0}^{\infty }w(t)\left( \mathcal{M}f(t)\right) ^{p}dt\leq
C\int_{0}^{\infty }w(t)f^{p}(t)dt.  \label{discrete hardy3}
\end{equation}%
has been proved for a nonnegative function $f\in L_{w}^{p}(\mathbb{R}^{+})$
and the constant $C$ depends of the norm $[w]_{A_{p}}$ of the weight $w.$
The weight $w$ is in the Muckenhoupt class $A_{1}$\ if there exists a
constant $\mathcal{C}>1$\ such that%
\begin{equation}
\mathcal{M}w(x)\leq \mathcal{C}w(x),\text{ \ for all }x\in \mathbb{R}^{+},
\label{1}
\end{equation}%
where the operator $f\rightarrow \mathcal{M}f$\ is the Hardy-Littlewood
maximal\ operator. The infimum of all such $\mathcal{C}$ is denoted by $%
[w]_{A_{1}}$ and is called the $A_{1}\emph{-}$norm.

Muckenhoupt's result became a landmark in the theory of weighted norm
inequalities because most of the previously known results for classical
operators had been obtained for special classes of weights (like power
weights) and has been extended to cover several operators like Hardy's
operator, Hilbert's operator, Calder\'{o}n-Zygmund's singular operator,
fractional integral operator. For more details about the properties of
Muckenhoupt weights and their applications, we refer the reader to the
papers\ \cite{car, Neu, Acta, Bull} and the references cited therein. By H%
\"{o}lder's inequality it is easy to prove that if $w$ satisfies the $A_{p}$
condition (\ref{Ap}) then $w$ satisfies the Muckenhoupt condition 
\begin{equation}
\left( \frac{1}{\left\vert I\right\vert }\int_{I}w(x)dx\right) \left( \frac{1%
}{\left\vert I\right\vert }\int_{I}w^{-\frac{1}{p_{1}-1}}(x)dx\right)
^{p_{1}-1}\leq \mathcal{C},\text{ for }p_{1}>p.
\end{equation}%
The main problem that has been solved in \cite{M} by Muckenhoupt is the
(backward propagation) self-improving property of the Muckenhoupt class of
weights which sates that: If $w\in A_{p}(\mathcal{C)}$ then there exists $%
\epsilon =\epsilon (p,\mathcal{C})>0$ such that%
\begin{equation}
\left( \frac{1}{\left\vert I\right\vert }\int_{I}w(t)dt\right) \left( \frac{1%
}{\left\vert I\right\vert }\int_{I}w^{-\frac{1}{q-1}}(t)dt\right) ^{q-1}\leq 
\mathcal{C}_{1},  \label{Api}
\end{equation}%
for all $q=p-\epsilon $, where the constant $\mathcal{C}_{1}=\mathcal{C}%
_{1}(p,\mathcal{C}).$ In other words, Muckenhoupt result for self-improving
property shows that if $w\in A_{q}(\mathcal{C})$ then there exists $\epsilon
>0$ such that $w\in A_{q-\epsilon }(\mathcal{C}_{1}),$ and then%
\begin{equation}
A_{q}(\mathcal{C})\subset A_{q-\epsilon }(\mathcal{C}_{1}).  \label{AP1}
\end{equation}%
In 1982 Rubio de Francia \cite{193} discovered a very different perspective
theorem on the theory of weighted norm inequalities. This theorem\ called\
the extrapolation theorem which studies the boundedness of operators in
different spaces via the Muckenhoupt weights. In particular, it has been
proved that for a given sublinear operator $\mathcal{T},$\ for some $p_{0}$
such that $1\leq p_{0}<\infty ,$ and every $w\in A_{p_{0}},$ there exists a
constant $\mathcal{C}>0$\ depends on $[w]_{A_{p_{0}}}$ such that%
\begin{equation*}
\int_{0}^{\infty }\left\vert \mathcal{T}f(x)\right\vert ^{p_{0}}w(x)dx\leq 
\mathcal{C(}[w]_{A_{p_{0}}})\int_{0}^{\infty }\left\vert f(x)\right\vert
^{p_{0}}w(x)dx,
\end{equation*}%
then for every $1<p<\infty $ and every $w\in A_{p},$ there exists a constant 
$\mathcal{C}_{1}$ depends on $[w]_{A_{p}}$\ such\ that%
\begin{equation*}
\int_{0}^{\infty }\left\vert \mathcal{T}f(x)\right\vert ^{p}w(x)dx\leq 
\mathcal{C}_{1}([w]_{A_{p}})\int_{0}^{\infty }\left\vert f(x)\right\vert
^{p}w(x)dx.
\end{equation*}%
The original proof of Rubio de Francia extrapolation theorem was quite
complex and depends on a connection between vector-valued estimates and a
weighted norm inequalities. A\ more direct proof that depend only on
weighted norm inequalities was given by Garci\'{a}-Cuerva \cite{10a}.
However this approach requires two complicated lemmas on the structure of $%
A_{p}-$ weights.

\bigskip

In \cite{extrapolation} the authors proved the\ extrapolation theorem\ with
sharp bounds with a different method depends on the applications of Rubio de
Francia algorithm. In \cite{ext11} Duoandikoetxea proved the\ extrapolation
theorem for Muckenhoupt weights with explicit bounds and avoided the
iteration algorithm when $p_{0}>1.$ The proof has two steps: first prove the
desired inequality holds for $1<p<p_{0}$ and $w\in A_{1}$ and then use this
to prove the full result. The proof also requires the very deep
self-improving property of $A_{p}$ weights that: if $w\in A_{p}$ then there
exists $\epsilon >0$ such that $w\in A_{p-\epsilon }.$

\bigskip

In \cite{Carro} Carro and Lorente proved a new version of Rubio de Francia
extrapolation theorem in the setting of $B_{p}-$weights, instead of the $%
A_{p}-$weights, for pairs of positive decreasing functions defined on $%
\mathbb{R}^{+}$. A weight $w$ is said to be belong to the class $B_{p}(C)$
for $0<p<\infty \ $if it satisfies the condition 
\begin{equation}
\int_{t}^{\infty }\frac{w(x)}{x^{p}}dx\leq \frac{C}{t^{p}}\int_{0}^{t}w(x)dx%
\text{, \ for all }t>0.  \label{B0}
\end{equation}%
The smallest constant $C>0$ satisfying (\ref{B0}) is called the $B_{p}-$%
constant of the class and is denoted by $\left[ B_{p}(w)\right] $. This
class has been introduced by\ Ari\v{n}o and Muckenhoupt \cite{Arino} in
connection with the boundedness of the Hardy operator $\mathcal{H}%
f(t)=(1/t)\int_{0}^{t}f(x)dx$, for $t>0$ on the space $L_{w}^{p}[\mathbb{R}%
^{+}).$ The proof also requires the very deep self-improving property of $%
B_{p}$ weights that: if $w\in B_{p}$ then there exists $\epsilon >0$ such
that $w\in B_{p-\epsilon }.$

The theory of Rubio de Francia extrapolation has found a host of
applications in harmonic analysis. One of the important applications is due
to Duoandikoetxea and Rubio de Francia \cite{71}. They have proved the
boundedness of the rough singular integrals on $L^{p}(w)$ with kernels via
the Muckenhoupt weights $w\in A_{p}$. The key to the proof was the
extrapolation theorem, since this reduced the problem to proving that the
singular operator is bounded on $L^{2}(w)$,$\ w\in A_{2}.$\ 

\bigskip

In the last years, there has been renewed interest in the area of discrete
harmonic analysis and then it becomes an active field of research. For
example, the study of discrete\ $\ell _{p}(w)$ analogues for\ $L^{p}(w)$
results has been considered by some authors, see for example \cite{3, 4, 5,
7, 19, PAMS, 23, 24, samir} and the references cited therein. In \cite{Benet}
the authors mentioned that the study of discrete norm inequalities is not an
easy task and more difficult to analyze than its integral counterparts and
discovered that the discrete conditions do not correspond, in any natural
way, with those that are obtained by discretization the results of functions
but the reverse is true. This means that what goes for sums goes, with the
obvious modifications, for integrals which in fact proved the first part of
basic principle of Hardy, Littlewood and Polya \cite[p. 11]{Hardy}.

Indeed the proofs for series translate immediately and become much simpler,
when applied to integrals. For illustration, we present the following
examples to show the big difference of discrete and continuous inequalities.
In \cite{krnvsk1} the author proved the inequality 
\begin{equation}
\int_{a}^{b}\frac{\lambda (t)}{\Lambda ^{1-\alpha }(t)}\left( \mathcal{M}%
v\left( t\right) \right) ^{\beta }dt\leq \left( \frac{\beta }{\beta -\alpha }%
\right) ^{\beta }\int_{a}^{b}\frac{\lambda (t)}{\Lambda ^{1-\alpha }(t)}%
v^{\beta }\left( t\right) dt,  \label{hk}
\end{equation}%
where 
\begin{equation*}
\mathcal{M}v(t):=\frac{1}{\Lambda (t)}\int_{a}^{t}\lambda (s)v\left(
s\right) ds\text{, and }\Lambda (t)=\int_{a}^{t}\lambda (s)ds.
\end{equation*}%
The discrete analogy of this inequality is given by (translation from
integral to series directly): Let $v$ be a positive sequence and $\lambda $
be a nonnegative sequence and $\Lambda \lbrack
1,n]=\sum\limits_{k=1}^{n}\lambda (k)$ for all $n\in \mathbb{Z}_{+}$.
Suppose that $0<\alpha <1$, $\beta <0$ or $\beta \geq 1$ and $\mathcal{M}v$
is defined by%
\begin{equation*}
\mathcal{M}v(k):=\frac{1}{\Lambda \lbrack 1,k]}\sum\limits_{s=1}^{k}\lambda
(s)v\left( s\right) .
\end{equation*}%
Then%
\begin{equation}
\sum\limits_{k=1}^{N}\frac{\lambda (k)}{\Lambda ^{1-\alpha }[1,k]}\left( 
\mathcal{M}v\left( k\right) \right) ^{\beta }\leq \left( \frac{\beta }{\beta
-\alpha }\right) ^{\beta }\sum\limits_{k=1}^{N}\frac{\lambda (k)}{\Lambda
^{1-\alpha }[1,k]}v^{\beta }\left( k\right) .  \label{KL1}
\end{equation}%
In the unweighted case $\lambda (k)=1$ and $\beta -\alpha =1$, the
inequality (\ref{KL1}) becomes 
\begin{equation}
\sum\limits_{k=1}^{N}\frac{1}{k^{2}}\left( \sum\limits_{s=1}^{k}v\left(
s\right) \right) ^{\beta }\leq \beta ^{\beta }\sum\limits_{k=1}^{N}\frac{1}{%
k^{1-\alpha }}v^{\beta }\left( k\right) .  \label{3}
\end{equation}%
Numerical (counter) example: Set in (\ref{3}) $N=4$, $\beta =1.2$, $\alpha
=0.2$, $v(1)=100$, $v(2)=v(3)=v)4)=1$, we see that the right side is given by%
\begin{equation*}
\sum\limits_{k=1}^{N}\frac{1}{k^{2}}\left( \sum\limits_{s=1}^{k}v\left(
s\right) \right) ^{\beta }=359.587
\end{equation*}%
and the left hand side is given by%
\begin{equation*}
\beta ^{\beta }\sum\limits_{k=1}^{N}\frac{1}{k^{1-\alpha }}v^{\beta }\left(
k\right) =314.263
\end{equation*}%
For the case $\beta <0$ and $\beta -\alpha =-1,$ then (\ref{KL1}) becomes
(unweighted case): 
\begin{equation}
\sum\limits_{k=1}^{N}\left( \sum\limits_{s=1}^{k}v\left( s\right) \right)
^{\beta }\leq (-\beta )^{\beta }\sum\limits_{k=1}^{N}\frac{1}{k^{1-\alpha }}%
v^{\beta }\left( k\right) .  \label{1.12}
\end{equation}%
By choosing $\beta =-0.9$, $\alpha =0.1$, $N=5,$ $v(1)=1,$ $v(2)=5$, $v(3)=9$%
, $v(4)=13,$ $v(5)=17.$ Numerical gives the right hand side 
\begin{equation*}
\sum\limits_{k=1}^{N}\left( \sum\limits_{s=1}^{k}v\left( s\right) \right)
^{\beta }=1.36913\text{, }
\end{equation*}%
and the left hand side%
\begin{equation*}
(-\beta )^{\beta }\sum\limits_{k=1}^{N}\frac{1}{k^{1-\alpha }}v^{\beta
}\left( k\right) =1.3406.
\end{equation*}%
This shows that (\ref{3})\ and (\ref{1.12}) are not valid and a constant $%
\zeta >1$ should be exists in the right hand side. This proves that the
discrete conditions do not correspond, in any natural way, with those that
are obtained by discretization the results of functions. So the inequality (%
\ref{KL1}) should be 
\begin{equation*}
\sum\limits_{k=1}^{N}\frac{\lambda (k)}{\Lambda ^{1-\alpha }[1,k]}\left( 
\mathcal{M}v\left( k\right) \right) ^{\beta }\leq \zeta \left( \frac{\beta }{%
\beta -\alpha }\right) ^{\beta }\sum\limits_{k=1}^{N}\frac{\lambda (k)}{%
\Lambda ^{1-\alpha }[1,k]}v^{\beta }\left( k\right) ,
\end{equation*}%
where $\Lambda \lbrack 1,k]\leq \zeta \Lambda \lbrack 1,k-1]$ for some $%
\zeta >1$ and all $k$. Another choice is to put the inequality in the form 
\begin{equation}
\sum\limits_{k=1}^{N}\frac{\lambda (k)}{\Lambda ^{1-\alpha }[1,k]}\left( 
\mathcal{M}v\left( k-1\right) \right) ^{\beta }\leq \left( \frac{\beta }{%
\beta -\alpha }\right) ^{\beta }\sum\limits_{k=1}^{N}\frac{\lambda (k)}{%
\Lambda ^{1-\alpha }[1,k]}v^{\beta }\left( k\right) .  \label{Ka1}
\end{equation}%
For the unweighted case ($\lambda (s)=1$ for every $s$) and $\beta -\alpha
=1 $, inequality (\ref{Ka1}) becomes:%
\begin{equation}
\sum_{k=1}^{N}\frac{1}{k^{1-\alpha }}\left( \frac{1}{k-1}%
\sum_{s=1}^{k-1}v(s)\right) ^{\beta }\leq \beta ^{\beta }\sum_{k=1}^{N}\frac{%
1}{k^{1-\alpha }}v^{\beta }(k).  \label{kka1}
\end{equation}%
The numerical gives us that: For $N=4,$ $\beta =1.2,$ $\alpha =0.2,$ $%
v_{1}=100,v_{2}=1,$ $v_{3}=1,$ $v_{4}=1$, then the left-side of (\ref{kka1})
gives:%
\begin{equation*}
\frac{1}{1^{0.8}}(0)^{1.2}+\frac{1}{2^{0.8}}\left( v_{1}\right) ^{1.2}+\frac{%
1}{3^{0.8}}\left( \frac{v_{1}+v_{2}}{2}\right) ^{1.2}+\frac{1}{4^{0.8}}%
\left( \frac{v_{1}+v_{2}+v_{3}}{3}\right) ^{1.2}=212.922,
\end{equation*}%
and the right-side gives:%
\begin{equation*}
1.2^{1.2}\ast (v_{1}^{1.2}+\frac{v_{2}^{1.2}}{2^{0.8}}+\frac{v_{3}^{1.2}}{%
3^{0.8}}+\frac{v_{4}^{1.2}}{4^{0.8}})=314.263.
\end{equation*}%
This shows that (\ref{Ka1}) holds. For the unweighted case and $\beta <0,$ $%
\beta -\alpha =-1$, inequality (\ref{Ka1}) becomes:%
\begin{equation}
\sum_{k=1}^{N}\frac{1}{k^{1-\alpha }}\left( \frac{1}{k-1}%
\sum_{s=1}^{k-1}v(s)\right) ^{\beta }\leq \left( -\beta \right) ^{\beta
}\sum_{k=1}^{N}\frac{1}{k^{1-\alpha }}v^{\beta }(k).  \label{kka2}
\end{equation}%
For $N=5,\beta =-0.9,$ $\alpha =0.1,$ $v_{1}=1,$ $v_{2}=5,$ $v_{3}=9,$ $%
v_{4}=13,$ $v_{5}=17$, the left-side of (\ref{kka2}) gives:%
\begin{equation*}
\frac{1}{1^{0.9}}(0)^{0.9}+\frac{1}{2^{0.9}}\left( v_{1}\right) ^{-0.9}+%
\frac{1}{3^{0.9}}\left( \frac{v_{1}+v_{2}}{2}\right) ^{-0.9}
\end{equation*}%
\begin{equation*}
+\frac{1}{4^{0.9}}\left( \frac{v_{1}+v_{2}+v_{3}}{3}\right) ^{-0.9}+\frac{1}{%
5^{0.9}}\left( \frac{v_{1}+v_{2}+v_{3}+v_{4}}{4}\right) ^{-0.9}=0.7743
\end{equation*}%
and the right-side gives:%
\begin{equation*}
(0.9)^{-0.9}\ast \left( v_{1}^{-0.9}+\frac{v_{2}^{-0.9}}{2^{0.9}}+\frac{%
v_{3}^{-0.9}}{3^{0.9}}+\frac{v_{4}^{-0.9}}{4^{0.9}}+\frac{v_{5}^{-0.9}}{%
5^{0.9}}\right) =1.3516.
\end{equation*}%
This shows that (\ref{Ka1}) holds and explains our aim to prove the discrete
results and do not translate the continuous results directly to the discrete
versions and also gives an answer to the questions from the referees to our
discrete papers.

In the following, for completeness, we present some of the related results
of the discrete Muckenhoupt $\mathcal{A}_{p}$ weights and discrete
extrapolation theorem of discrete operators to show the motivation of this
paper. Throughout the paper, we assume that $1<p<\infty $ and assume that $J%
\mathbb{\subset Z}_{+}=\{1,2,3,...\},$ where $J$ is of the form $%
J=\{1,2,\ldots ,n\}$. A discrete weight $w$\ is a sequence $\left\{
w(n)\right\} _{n=1}^{\infty }$\ of nonnegative real numbers. Most often $w$
will appear in the role of the weight in $\ell _{p}(w)-$estimates, i.e. we
shall consider the norm 
\begin{equation*}
\left\Vert f\right\Vert _{\ell _{p}(w)}:=\left( \sum_{k=1}^{\infty
}w(k)\left\vert f(k)\right\vert ^{p}\right) ^{\frac{1}{p}}<\infty .
\end{equation*}%
A discrete nonnegative weight $w$ defined on $\mathbb{Z}_{+}=\{1,2,\ldots \}$
belongs to the discrete Muckenhoupt class $\mathcal{A}_{1}(C)$ for $p>1$ and 
$C>1$ if the inequality 
\begin{equation}
\left( \frac{1}{n}\sum\limits_{k=1}^{n}{w(k)}\right) {\left( \frac{1}{n}%
\sum\limits_{k=1}^{n}w^{\frac{-1}{p-1}}(k)\right) ^{p-1}}\leq C,  \label{Ap1}
\end{equation}%
holds for every $n>1.$ For a given exponent $p>1$ we define the $\mathcal{A}%
_{p}(A)-$norm of the discrete weight $w$ by the following quantity%
\begin{equation*}
\left[ w\right] _{\mathcal{A}_{p}}:=\sup_{n\geq 1}\frac{1}{n}%
\sum\limits_{k=1}^{n}w(n)\left( \frac{1}{n}\sum\limits_{k=1}^{n}w^{\frac{1}{%
1-p}}(n)\right) ^{p-1}.
\end{equation*}%
The boundedness of discrete Hardy-Littlewood maximal operator 
\begin{equation}
\mathcal{M}f(n):=\sup_{n>1}\frac{1}{n}\sum\limits_{k=1}^{n}f(k)\text{,}
\label{3.1}
\end{equation}%
where $f(n)$ is nonnegative sequence has been characterized in \cite%
{Saker2020} in terms of the Muckenhoupt weights $\mathcal{A}_{p}$. They
proved that $\mathcal{M}f$ is bounded in $\ell _{p}(w)$ iff $w\in \mathcal{A}%
_{p}$ and 
\begin{equation*}
\left\Vert \mathcal{M}f\right\Vert _{\ell ^{p}(w)}\leq C\left\Vert
f\right\Vert _{\ell ^{p}(w)},
\end{equation*}%
where the constant $C$ depends on the norm $[w]_{\mathcal{A}_{p}}.$ The
discrete weight $w$ belongs to the discrete Muckenhoupt class $\mathcal{A}%
_{1}$ if there exists a constant $C>1$\ such that 
\begin{equation*}
\mathcal{M}w(n)\leq Cw(n),\text{ \ \ \ }n\in \mathbb{Z}_{+},
\end{equation*}%
where $\mathcal{M}$ is the Hardy-Littlewood maximal operator. The infimum of
all such $C$ is denoted by $[w]_{\mathcal{A}_{1}}$ and is called the $%
\mathcal{A}_{1}\emph{-}$norm. A discrete weight\ $w$\ is said to be belong
to\ the discrete Muckenhoupt class\ $\mathcal{A}_{\infty }$ if\ there exists
a constant $C>1$ such that%
\begin{equation*}
\left( \frac{1}{n}\sum\limits_{k=1}^{n}w\right) \left( \exp \frac{1}{n}%
\sum\limits_{k=1}^{n}\log \frac{1}{w}\right) \leq C.
\end{equation*}%
The infimum of all such $C$ is denoted by $[w]_{\mathcal{A}_{\infty }}$ and
is called the $\mathcal{A}_{\infty }\emph{-}$norm. The prototypical $%
\mathcal{A}_{p}-$weights are the power weights. The authors in \cite{Saker
et al} (see also \cite{Saker2023}) investigated the following estimates for
power low discrete weights. If $p>1$ and $-1<\lambda <p-1$, then $n^{\lambda
}\in A_{p}$ and its norm is given by $\left[ n^{\lambda }\right] _{\mathcal{A%
}_{p}}\simeq \Phi (p,\lambda )$, where 
\begin{equation}
\Phi (p,\lambda )=\frac{1}{(1+\lambda )}\left( \frac{p-1}{p-\lambda -1}%
\right) ^{p-1}  \label{a1}
\end{equation}%
Taking limit as $p\rightarrow 1,$ we get the following: If $-1<\lambda <0$,
then $n^{\lambda }\in \mathcal{A}_{1}$ and the norm $\left[ n^{\lambda }%
\right] _{\mathcal{A}_{1}}\simeq \frac{1}{(1+\lambda )}$.

In \cite{saker} Saker and Agarwal \ proved the discrete Rubio de Francia
extrapolation theorem\ via the discrete\ Muckenhoupt weights. Precisely they
proved that if\ $\varphi $ is an increasing function on $(0,\infty )$ and $%
\mathcal{T}$ is a sublinear operator defined on $\mathbb{Z}_{+}$ such that%
\begin{equation*}
\sum_{k=1}^{\infty }\left\vert \mathcal{T}f(k)\right\vert ^{p_{0}}w(k)\leq
\varphi \left( \left[ w\right] _{\mathcal{A}_{p_{0}}}\right)
\sum_{k=1}^{\infty }\left\vert f(k)\right\vert ^{p_{0}}w(k),
\end{equation*}%
for some $p_{0},$ $1\leq p_{0}<\infty ,$ and every weight\ $w\in \mathcal{A}%
_{p_{0}},$ then 
\begin{equation*}
\sum_{k=1}^{\infty }\left\vert \mathcal{T}f(k)\right\vert ^{p}w(k)\leq
\varphi \left( \left[ w\right] _{\mathcal{A}_{p}}\right) \sum_{k=1}^{\infty
}\left\vert f(k)\right\vert ^{p}w(k),
\end{equation*}%
for every $p$ with\ $1<p<\infty $ and every $w\in \mathcal{A}_{p}.$ The
proof depends on the weighted norm inequality 
\begin{equation}
\left\Vert \mathcal{M}^{w}f\right\Vert _{\ell _{w}^{r}\left( \mathbb{Z}%
_{+}\right) }\leq C\left\Vert f\right\Vert _{\ell _{w}^{r}\left( \mathbb{Z}%
_{+}\right) },  \label{dg}
\end{equation}%
for $r>1$ and $f$ be a nonnegative series defined on $\mathbb{Z}_{+}$, where 
$\mathcal{M}^{w}f$ is the weighted Hardy-Littlewood maximal operator 
\begin{equation}
\left[ \mathcal{M}^{w}\left( f\right) \right] (k)=\sup_{k\in J}\frac{1}{%
\Lambda (J)}\sum_{s\in J}w(s)f(s),\text{ for }k\in \mathbb{Z}_{+},
\label{1.2q}
\end{equation}%
where $\Lambda (J)=\sum_{s\in J}w(s)$. However this approach requires two
complicated lemmas on the structure of $\mathcal{A}_{p}-$ weights and the
self-improving property of the discrete $\mathcal{A}_{p}$ Muckenhoupt
weights: if $v\in \mathcal{A}_{q}(C)$ then there exists an $\epsilon >0$ and
a constant $C_{1}=C_{1}(p,$ $C)$ such that $v\in \mathcal{A}_{q-\epsilon
}(C_{1})$.

In \cite{saker1} Saker and Agarwal proved the discrete Rubio de Francia
extrapolation theorem\ via\ discrete $\mathcal{B}_{p}$ weights by employing
the self-improving property of the discrete $\mathcal{B}_{p}$ weights that:
if $v\in \mathcal{B}_{q}(C)$ then there exists an $\epsilon >0$ and a
constant $C_{1}=C_{1}(q,$ $C)$ such that $v\in \mathcal{B}_{q-\epsilon
}(C_{1})$. A nonnegative discrete sequence $w$ defined on $\mathbb{Z}_{+}$
is said to be belong to the discrete class $\mathcal{B}_{p}(B)$ for $p>0$
and $B>0$ if $w$ satisfies the condition 
\begin{equation}
\sum_{k=n}^{\infty }\frac{w(k)}{k^{p}}\leq \frac{B}{n^{p}}\sum_{k=1}^{n}w(n),%
\text{ \ for all }n\in \mathbb{Z}_{+}.  \label{B}
\end{equation}%
In \cite{Benet} Bennett and Grosse-Erdmann proved\ that the discrete Hardy
operator 
\begin{equation*}
\mathcal{H}f(n)=\frac{1}{n}\sum_{k=1}^{n}f(k),
\end{equation*}%
is bounded\ on $\ell _{p}(w)$ for all decreasing sequence $f$ if and only if 
$w\in \mathcal{B}_{p}(B).$

\bigskip

In this paper, we give a new proof of the discrete extrapolation theorem
via\ Muckenhoupt weights $\mathcal{A}_{p}$ with explicit constants$.$ The
proof is different from the proof in \cite{saker} and simpler and more
direct, since it yields the desired inequality directly without cases of
intermediate steps and uses the iteration algorithm 
\begin{equation*}
\mathcal{N}\left( h(n)\right) :=\sum_{s=0}^{\infty }\frac{\mathcal{M}^{s}h(n)%
}{2^{s}\left\Vert \mathcal{M}\right\Vert _{\ell _{p}(w)}^{s}},
\end{equation*}%
where\ for $s>0,$ $\mathcal{M}^{s}h=\mathcal{M\circ }...$ $\mathcal{\circ M}%
h $ denotes $s$ iterations of the Hardy-Littlewood maximal operator and $%
\mathcal{M}^{0}h=h,$ and the dual iteration algorithm 
\begin{equation}
\mathcal{N}^{\prime }h(n):=\sum_{s=0}^{\infty }\frac{\left( \mathcal{M}%
^{\prime }\right) ^{s}h(n)}{2^{s}\left\Vert \mathcal{M}^{\prime }\right\Vert
_{\ell _{p^{\prime }}(w)}^{s}},
\end{equation}%
with%
\begin{equation}
\mathcal{M}^{\prime }h(n)=\frac{\mathcal{M}\left( wh\right) (n)}{w(n)},\text{
\ }n\in \mathbb{Z}_{+}\text{,}
\end{equation}%
where\ $\mathcal{M}h$ is the Hardy-Littlewood maximal operator and $\left( 
\mathcal{M}^{\prime }\right) ^{s}h=\mathcal{M}^{\prime }\mathcal{\circ }...$ 
$\mathcal{\circ M}^{\prime }h$ denotes $s$ iterations of the operator\ $%
\mathcal{M}^{\prime },$\ and\ $\left( \mathcal{M}^{\prime }\right) ^{0}h=h.$
In addition, we will use some of the properties of $\mathcal{A}_{p}$ weights
that are presented in the following lemmas for the sake of completeness.

\begin{lemma}
\cite{samir} Assume that $w$\ is a nonnegative weight and $1<p,$ $q<\infty \ 
$ are positive real numbers. Then the following properties hold.

$(i).$ If $w\in \mathcal{A}_{p},$ then $w^{\alpha }\in \mathcal{A}_{p}$ for\ 
$0\leq \alpha \leq 1$ with $\left[ w^{\alpha }\right] _{\mathcal{A}_{p}}\leq
\left( \left[ w\right] _{\mathcal{A}_{p}}\right) ^{\alpha }$.

$(ii).$ $w\in \mathcal{A}_{p}$\ if and only if $w^{1-p^{\prime }}\in 
\mathcal{A}_{p^{\prime }},$ $p^{\prime }=p/\left( p-1\right) $ and 
\begin{equation}
\ \left[ w^{1-p^{\prime }}\right] _{\mathcal{A}_{p^{\prime }}}=\left[ w%
\right] _{\mathcal{A}_{p}}^{p^{\prime }-1}.  \label{t*}
\end{equation}
\end{lemma}

\begin{lemma}
\cite{samir} Let $w,$ $w_{1},$ $w_{2}$ be discrete\ weights. Then the
following properties hold.

$1)$ $w\in \mathcal{A}_{p},$ $1<p<\infty $ if and only if there exist $%
w_{1}, $ $w_{2}\in \mathcal{A}_{1}$ such that $w=w_{1}w_{2}^{1-p}.$

$2)$ If $w_{1},$\ $w_{2}\in \mathcal{A}_{p},$ then $w_{1}^{\alpha
}w_{2}^{1-\alpha }\in \mathcal{A}_{p},$ $0<\alpha \leq 1$ with a constant \
\ 
\begin{equation*}
\left[ w_{1}^{\alpha }w_{2}^{1-\alpha }\right] _{\mathcal{A}_{p}}=\left[
w_{1}\right] _{\mathcal{A}_{p}}^{\alpha }\left[ w_{2}\right] _{\mathcal{A}%
_{p}}^{1-\alpha }.
\end{equation*}

$3).$ $\mathcal{A}_{p}\subset \mathcal{A}_{q}$ for all $1<p\leq q<\infty \ $%
and\ $\left[ w\right] _{\mathcal{A}_{q}}\leq \left[ w\right] _{\mathcal{A}%
_{p}}.$
\end{lemma}

The paper is organized as follows: In Section $2,$ we present the basic\
properties of the Rubio de Francia iterated algorithm proved in \cite{saker}%
. Next, we prove the properties of the Dual Rubio de Francia iterated
algorithm by using the result proved in \cite{saker}. In Section $3,$ we
prove some weighted norm inequalities of pairs of sequences $(f,g)$ by
employing the Rubio de Francia iterated algorithms and the reverse
factorization properties of\ Muckenhoupt\ weights $\mathcal{A}_{p}$. We
consider two cases: the first case when $1<p<p_{0}<\infty $ and the second
case when $1<p_{0}<p<\infty ,$ where $p_{0}$ is a fixed number,\ $%
1<p_{0}<\infty .$ The results will be extended to the extrapolation theorem
via $\mathcal{A}_{\infty }$ weights. The proof of the extrapolation theorem
with exact constants will be obtained by combining the two cases and
replacing the pair $(f,g)$ by $(\mathcal{T}f,f).$ The advantage of this
approach is that a number of different results become special cases of the
extrapolation theorem with exact constants. The case when $0<p<1$ will be
considered in future and also left to the interested readers.

\bigskip

The features of the new proof of the discrete Rubio de Francia extrapolation
are the combining of some interesting results: norm inequalities for the
Hardy-Littlewood maximal operator $\mathcal{M}$, duality $\mathcal{M}%
^{^{\prime }}$, and the reverse factorization of $\mathcal{A}_{p}-$weights.
In the proof we also do not assume that the operator $\mathcal{T}$ is linear
or sublinear as posed in \cite{saker}. More precisely, we are talking about
the properties:

$1).$ $\mathcal{M}$ is sublinear, positive and bounded in $\ell _{p}(w)$ if $%
w\in \mathcal{A}_{p}$ (\cite{Saker2020})

$2).$ $\mathcal{M}^{^{\prime }}$ is sublinear, positive and bounded in $\ell
_{p^{^{\prime }}}(w)$ if $w\in \mathcal{A}_{p}$ (will be proved later).

$3).$ If $w_{1},$ $w_{2}\in \mathcal{A}_{1}$ then$w_{1}w_{2}^{1-p}\in 
\mathcal{A}_{p}$ (see Lemma 1.2).

$4).$ $w\in \mathcal{A}_{p}$\ if and only if $w^{1-p^{\prime }}\in \mathcal{A%
}_{p^{\prime }},$ $p^{\prime }=p/\left( p-1\right) $ (see Lemma 1.1)

\section{Preliminaries and Basic Lemmas}

In \cite{Saker2020} the authors proved that the Hardy-Littlewood maximal
operator $\mathcal{M}$ is bounded in $\ell _{p}(w)$ if $w$ is a Muckenhoupt
weight. In \cite{saker} the authors used this result and proved the
following results which will play crucial roles in proving the main results
in Section 3.

\begin{theorem}
\label{s} Assume that $0<\gamma \leq 1$ and $w\in \mathcal{A}_{p}$ for $%
1<p<\infty .$\ For a nonnegative sequence $g,$ we define%
\begin{equation*}
Gg:=\left( \frac{\mathcal{M}\left( g^{1/\gamma }w\right) }{w}\right)
^{\gamma }.
\end{equation*}%
Then the operator $G$ is bounded in $\ell _{p^{\prime }/\gamma }(w),$ $%
p^{\prime }=p/\left( p-1\right) $ and there exists a constant $C>0$ such that%
\begin{equation*}
\left\Vert Gg\right\Vert _{\ell _{p^{\prime }/\gamma }(w)}\leq C\left\Vert
g\right\Vert _{\ell _{p^{\prime }/\gamma }(w)}.
\end{equation*}
\end{theorem}

\begin{lemma}
\label{L}(Discrete Rubio de Francia algorithm) Fix $1<p<\infty $ and\ assume
that $w\in \mathcal{A}_{p}.$ For any nonnegative sequence\ $h\in \ell
_{p}(w) $, we define 
\begin{equation*}
\mathcal{N}h(n):=\sum_{s=0}^{\infty }\frac{\mathcal{M}^{s}h(n)}{%
2^{s}\left\Vert \mathcal{M}\right\Vert _{\ell _{p}(w)}^{s}},
\end{equation*}%
where\ for $s>0,$ $\mathcal{M}^{s}h=\mathcal{M\circ }...$ $\mathcal{\circ M}%
h $ denotes $s$ iterations of the Hardy-Littlewood maximal operator and $%
\mathcal{M}^{0}h=h.$ Then

$(i)$ $h(n)\leq \mathcal{N}h(n),$

$(ii)$ $\left\Vert \mathcal{N}h\right\Vert _{\ell _{p}(w)}\leq 2\left\Vert
h\right\Vert _{\ell _{p}(w)},$

$(iii)$ $\mathcal{N}h\in \mathcal{A}_{1}$ and $\left[ \mathcal{N}h\right] _{%
\mathcal{A}_{1}}\leq 2\left\Vert \mathcal{M}\right\Vert _{\ell _{p}(w)}.$
\end{lemma}

Now, we define the dual Discrete Rubio de Francia algorithm and prove the
properties of it.

\begin{lemma}
\label{s1}Assume that $w\in \mathcal{A}_{p},$\ for $1<p<\infty .$\ For a
nonnegative sequence $h,$ we define a general\ Hardy-Littlewood maximal\
operator by 
\begin{equation*}
\mathcal{M}^{\prime }h:=\frac{\mathcal{M}\left( wh\right) }{w}.
\end{equation*}%
Then the operator $\mathcal{M}^{\prime }$\ is bounded in $\ell _{p^{\prime
}}(w),$ $p^{\prime }=p/\left( p-1\right) $ and there exists a constant $C>0$
such that%
\begin{equation}
\left\Vert \mathcal{M}^{\prime }h\right\Vert _{\ell _{p^{\prime }}(w)}\leq
C\left\Vert h\right\Vert _{\ell _{p^{\prime }}(w)}.  \label{f}
\end{equation}
\end{lemma}

\begin{proof}
Applying Theorem \ref{s} with $g=h$ and $\gamma =1,$ we see that the
operator $\mathcal{M}^{\prime }$ which is defined by 
\begin{equation*}
\mathcal{M}^{\prime }h:=\frac{\mathcal{M}\left( wh\right) }{w},
\end{equation*}%
is bounded in $\ell _{^{p^{\prime }}}(w)$ and satisfies that%
\begin{equation*}
\left\Vert \mathcal{M}^{\prime }h\right\Vert _{\ell _{p^{\prime }}(w)}\leq
C\left\Vert h\right\Vert _{\ell _{p^{\prime }}(w)},
\end{equation*}%
which is the desired inequality (\ref{f}). The proof is complete.
\end{proof}

\begin{lemma}
\label{L1}(Discrete Dual Rubio de Francia algorithm) Fix $1<p<\infty $\ and\ 
$w\in \mathcal{A}_{p}.$ For any nonnegative sequence\ $h\in \ell _{p^{\prime
}}(w),$ $p^{\prime }=p/(p-1),$\ we define 
\begin{equation}
\mathcal{N}^{\prime }h(n):=\sum_{s=0}^{\infty }\frac{\left( \mathcal{M}%
^{\prime }\right) ^{s}h(n)}{2^{s}\left\Vert \mathcal{M}^{\prime }\right\Vert
_{\ell _{p^{\prime }}(w)}^{s}},  \label{r}
\end{equation}%
with%
\begin{equation}
\mathcal{M}^{\prime }h(n)=\frac{\mathcal{M}\left( wh\right) (n)}{w(n)},\text{
\ }n\in \mathbb{Z}_{+}\text{,}  \label{r1}
\end{equation}%
where\ $\mathcal{M}h$ is the Hardy-Littlewood maximal operator and $\left( 
\mathcal{M}^{\prime }\right) ^{s}h=\mathcal{M}^{\prime }\mathcal{\circ }...$ 
$\mathcal{\circ M}^{\prime }h$ denotes $s$ iterations of the operator\ $%
\mathcal{M}^{\prime },$\ and\ $\left( \mathcal{M}^{\prime }\right) ^{0}h=h.$%
\ Then

$\,(a).$ $h(n)\leq \mathcal{N}^{\prime }h(n),$

$(b).$ $\left\Vert \mathcal{N}^{\prime }h\right\Vert _{\ell _{p^{\prime
}}(w)}\leq 2\left\Vert h\right\Vert _{\ell _{p^{\prime }}(w)},$

$(c).$ $\left( \mathcal{N}^{\prime }h\right) w\in \mathcal{A}_{1}$ and $%
\left[ \left( \mathcal{N}^{\prime }h\right) w\right] _{\mathcal{A}_{1}}\leq
2\left\Vert \mathcal{M}^{\prime }\right\Vert _{\ell _{p^{\prime }}(w)}.$
\end{lemma}

\begin{proof}
From (\ref{r}), we\ have (where $\left( \mathcal{M}^{\prime }\right) ^{0}h=h$%
) that%
\begin{equation*}
\mathcal{N}^{\prime }h(n)=h(n)+\sum_{s=1}^{\infty }\frac{\left( \mathcal{M}%
^{\prime }\right) ^{s}h(n)}{2^{s}\left\Vert \mathcal{M}^{\prime }\right\Vert
_{\ell _{p^{\prime }}(w)}^{s}},
\end{equation*}%
and then%
\begin{equation*}
h(n)\leq \mathcal{N}^{\prime }h(n),
\end{equation*}%
which proves the property $(a).$ Applying Minkowski's inequality, on the
term $\left\Vert \mathcal{N}^{\prime }h\right\Vert _{\ell _{p^{\prime
}}(w)}, $ we see that%
\begin{eqnarray}
\left\Vert \mathcal{N}^{\prime }h\right\Vert _{\ell _{p^{\prime }}(w)}
&=&\left( \sum_{n=1}^{\infty }w(n)\left( \mathcal{N}^{\prime }h(n)\right)
^{p^{\prime }}\right) ^{\frac{1}{p^{\prime }}}  \notag \\
&=&\left( \sum_{n=1}^{\infty }w(n)\left( \sum_{s=0}^{\infty }\frac{\left( 
\mathcal{M}^{\prime }\right) ^{s}h(n)}{2^{s}\left\Vert \mathcal{M}^{\prime
}\right\Vert _{\ell _{p^{\prime }}(w)}^{s}}\right) ^{p^{\prime }}\right) ^{%
\frac{1}{p^{\prime }}}  \notag \\
&\leq &\sum_{s=0}^{\infty }\frac{1}{2^{s}\left\Vert \mathcal{M}^{\prime
}\right\Vert _{\ell _{p^{\prime }}(w)}^{s}}\left( \sum_{n=1}^{\infty
}w(n)\left( \left( \mathcal{M}^{\prime }\right) ^{s}h(n)\right) ^{p^{\prime
}}\right) ^{\frac{1}{p^{\prime }}}  \notag \\
&=&\sum_{s=0}^{\infty }\frac{1}{2^{s}\left\Vert \mathcal{M}^{\prime
}\right\Vert _{\ell _{p^{\prime }}(w)}^{s}}\left\Vert \left( \mathcal{M}%
^{\prime }\right) ^{s}h\right\Vert _{\ell _{p^{\prime }}(w)}.  \label{d}
\end{eqnarray}%
From Lemma \ref{s1},\ since the operator $\mathcal{M}^{\prime }$ is bounded,
the inequality (\ref{d}) becomes 
\begin{equation*}
\left\Vert \mathcal{N}^{\prime }h\right\Vert _{\ell _{p^{\prime }}(w)}\leq
\sum_{s=0}^{\infty }\frac{1}{2^{s}}\left\Vert h\right\Vert _{\ell
_{p^{\prime }}(w)}=2\left\Vert h\right\Vert _{\ell _{p^{\prime }}(w)},
\end{equation*}%
which proves the property $(b).$ Since the operator $h\rightarrow \mathcal{M}%
^{\prime }h$\ is sublinear, then\ 
\begin{eqnarray*}
\mathcal{M}^{\prime }\left( \mathcal{N}^{\prime }h(n)\right) &=&\mathcal{M}%
^{\prime }\left( \sum_{s=0}^{\infty }\left( \frac{\left( \mathcal{M}^{\prime
}\right) ^{s}h(n)}{2^{s}\left\Vert \mathcal{M}^{\prime }\right\Vert _{\ell
_{p^{\prime }}(w)}^{s}}\right) \right) \\
&\leq &\sum_{s=0}^{\infty }\mathcal{M}^{\prime }\left( \frac{\left( \mathcal{%
M}^{\prime }\right) ^{s}h(n)}{2^{s}\left\Vert \mathcal{M}^{\prime
}\right\Vert _{\ell _{p^{\prime }}(w)}^{s}}\right) =\sum_{s=0}^{\infty }%
\frac{1}{2^{s}\left\Vert \mathcal{M}^{\prime }\right\Vert _{\ell _{p^{\prime
}}(w)}^{s}}\left( \mathcal{M}^{\prime }\right) ^{s+1}h(n) \\
&\leq &2\left\Vert \mathcal{M}^{\prime }\right\Vert _{\ell _{p^{\prime
}}(w)}\sum_{s=0}^{\infty }\frac{\left( \mathcal{M}^{\prime }\right) ^{s}h(n)%
}{2^{s}\left\Vert \mathcal{M}^{\prime }\right\Vert _{\ell _{p^{\prime
}}(w)}^{s}}\leq 2\left\Vert \mathcal{M}^{\prime }\right\Vert _{\ell
_{p^{\prime }}(w)}\mathcal{N}^{\prime }h(n),
\end{eqnarray*}%
and then by using (\ref{r1}), we have that%
\begin{equation*}
\mathcal{M}\left( w\mathcal{N}^{\prime }h\right) (n)\leq 2\left\Vert 
\mathcal{M}^{\prime }\right\Vert _{\ell _{p^{\prime }}(w)}\left( w\mathcal{N}%
^{\prime }h\right) (n).
\end{equation*}%
This proves that $w\mathcal{N}^{\prime }h\in \mathcal{A}_{1}$ with the norm 
\begin{equation*}
\left[ \left( \mathcal{N}^{\prime }h\right) w\right] _{\mathcal{A}_{1}}\leq
2\left\Vert \mathcal{M}^{\prime }\right\Vert _{\ell _{p^{\prime }}(w)},
\end{equation*}%
which proves the property $(c).$ The proof is complete.
\end{proof}

\begin{lemma}
\label{L*}Assume that\ $w,$\ $h$\ are\ positive weights$\ $and $%
1<p<p_{0}<\infty .$ If $w\in \mathcal{A}_{p},$ then\ $w\left( \mathcal{N}%
h\right) ^{-\left( p_{0}-p\right) }\in \mathcal{A}_{p_{0}}$ and 
\begin{equation}
\left[ w\mathcal{N}h^{-\left( p_{0}-p\right) }\right] _{\mathcal{A}%
_{p_{0}}}\leq \left( \left[ \mathcal{N}h\right] _{\mathcal{A}_{1}}\right)
^{p_{0}-p}\left[ w\right] _{\mathcal{A}_{p}},  \notag
\end{equation}%
where the operator $\mathcal{N}$ is defined as in Lemma \ref{L}.
\end{lemma}

\begin{proof}
From\ the property\ $(iii)$\ of Lemma \ref{L},\ we see that $\mathcal{N}h\in 
\mathcal{A}_{1},$ and then 
\begin{equation*}
\frac{1}{n}\sum\limits_{k=1}^{n}\mathcal{N}h(k)\leq \left[ \mathcal{N}h%
\right] _{\mathcal{A}_{1}}\mathcal{N}h(k),
\end{equation*}%
and thus we have for $p_{0}>p$\ that%
\begin{eqnarray}
&&\frac{1}{n}\sum\limits_{k=1}^{n}w(k)\mathcal{N}h^{-\left( p_{0}-p\right)
}(k)  \notag \\
&\leq &\frac{1}{n}\sum\limits_{k=1}^{n}w(k)\left( \frac{1}{n}%
\sum\limits_{k=1}^{n}\mathcal{N}h(k)\right) ^{-\left( p_{0}-p\right) }\left( %
\left[ \mathcal{N}h\right] _{\mathcal{A}_{1}}\right) ^{p_{0}-p}  \notag \\
&=&\left( \left[ \mathcal{N}h\right] _{\mathcal{A}_{1}}\right)
^{p_{0}-p}\left( \frac{1}{n}\sum\limits_{k=1}^{n}w(k)\right) \left( \frac{1}{%
n}\sum\limits_{k=1}^{n}\mathcal{N}h(k)\right) ^{-\left( p_{0}-p\right) }.
\label{b}
\end{eqnarray}%
Applying H\"{o}lder's inequality on the term%
\begin{equation*}
\left( \frac{1}{n}\sum\limits_{k=1}^{n}\left[ w(k)\left( \mathcal{N}h\right)
^{-\left( p_{0}-p\right) }(k)\right] ^{\frac{-1}{p_{0}-1}}\right) ^{p_{0}-1},
\end{equation*}%
with $\gamma =\left( p_{0}-1\right) /\left( p-1\right) >1$ and $\nu =\left(
p_{0}-1\right) /\left( p_{0}-p\right) $ ($1/\gamma +1/\nu =1$), we get that 
\begin{eqnarray}
&&\left( \frac{1}{n}\sum\limits_{k=1}^{n}\left[ w(k)\left( \mathcal{N}%
h\right) ^{-\left( p_{0}-p\right) }(k)\right] ^{\frac{-1}{p_{0}-1}}\right)
^{p_{0}-1}  \notag \\
&=&\left( \frac{1}{n}\right) ^{p_{0}-1}\left( \sum\limits_{k=1}^{n}w^{\frac{%
-1}{p_{0}-1}}(k)\left( \mathcal{N}h\right) ^{\frac{p_{0}-p}{p_{0}-1}%
}(k)\right) ^{p_{0}-1}  \notag \\
&\leq &\left( \frac{1}{n}\right) ^{p_{0}-1}\left( \sum\limits_{k=1}^{n}%
\mathcal{N}h(k)\right) ^{p_{0}-p}\left( \sum\limits_{k=1}^{n}w^{\frac{-1}{p-1%
}}(k)\right) ^{p-1}  \notag \\
&=&\left( \frac{1}{n}\sum\limits_{k=1}^{n}\mathcal{N}h(k)\right)
^{p_{0}-p}\left( \frac{1}{n}\sum\limits_{k=1}^{n}w^{\frac{-1}{p-1}%
}(k)\right) ^{p-1}.  \label{b1}
\end{eqnarray}%
From (\ref{b}) and (\ref{b1}), since $w\in \mathcal{A}_{p},$ $1<p<\infty ,$
we have that 
\begin{eqnarray*}
&&\left( \frac{1}{n}\sum\limits_{k=1}^{n}w(k)\left( \mathcal{N}h\right)
^{-\left( p_{0}-p\right) }(k)\right) \left( \frac{1}{n}\sum\limits_{k=1}^{n}%
\left[ w(k)\left( \mathcal{N}h\right) ^{-\left( p_{0}-p\right) }(k)\right] ^{%
\frac{-1}{p_{0}-1}}\right) ^{p_{0}-1} \\
&\leq &\left( \left[ \mathcal{N}h\right] _{\mathcal{A}_{1}}\right)
^{p_{0}-p}\left( \frac{1}{n}\sum\limits_{k=1}^{n}w(k)\right) \left( \frac{1}{%
n}\sum\limits_{k=1}^{n}\mathcal{N}h(k)\right) ^{-\left( p_{0}-p\right) } \\
&&\times \left( \frac{1}{n}\sum\limits_{k=1}^{n}\mathcal{N}h(k)\right)
^{p_{0}-p}\left( \frac{1}{n}\sum\limits_{k=1}^{n}w^{\frac{-1}{p-1}%
}(k)\right) ^{p-1} \\
&=&\left( \left[ \mathcal{N}h\right] _{\mathcal{A}_{1}}\right)
^{p_{0}-p}\left( \frac{1}{n}\sum\limits_{k=1}^{n}w(k)\right) \left( \frac{1}{%
n}\sum\limits_{k=1}^{n}w^{\frac{-1}{p-1}}(k)\right) ^{p-1}.
\end{eqnarray*}%
Taking supremum, we have that $w\left( \mathcal{N}h\right) ^{-\left(
p_{0}-p\right) }\in \mathcal{A}_{p_{0}}$ with the constant 
\begin{eqnarray*}
\left[ w\left( \mathcal{N}h\right) ^{-\left( p_{0}-p\right) }\right] _{%
\mathcal{A}_{p_{0}}} &=&\sup_{n}\left( \frac{1}{n}\sum\limits_{k=1}^{n}w(k)%
\left( \mathcal{N}h\right) ^{-\left( p_{0}-p\right) }(k)\right) \\
&&\times \left( \frac{1}{n}\sum\limits_{k=1}^{n}\left[ w(k)\left( \mathcal{N}%
h\right) ^{-\left( p_{0}-p\right) }(k)\right] ^{\frac{-1}{p_{0}-1}}\right)
^{p_{0}-1} \\
&\leq &\left( \left[ \mathcal{N}h\right] _{\mathcal{A}_{1}}\right) ^{p_{0}-p}%
\left[ w\right] _{\mathcal{A}_{p}},
\end{eqnarray*}

which is the desired inequality. The proof is complete.
\end{proof}

\begin{lemma}
\label{L1*}Assume that\ $w,$ $h$ are positive sequences$\ $on $\mathbb{Z}%
_{+}\ $and $1<p_{0}<p<\infty .$\ If\ $w\in \mathcal{A}_{p},$\ then $w\left( 
\mathcal{N}^{\prime }h\right) ^{\frac{p-p_{0}}{p-1}}\in \mathcal{A}_{p_{0}}$
and%
\begin{equation}
\left[ w\left( \mathcal{N}^{\prime }h\right) ^{\frac{p-p_{0}}{p-1}}\right] _{%
\mathcal{A}_{p_{0}}}\leq \left[ \left( \mathcal{N}^{\prime }h\right) w\right]
_{\mathcal{A}_{1}}^{\frac{p-p_{0}}{p-1}}\left[ w\right] _{\mathcal{A}_{p}}^{%
\frac{p_{0}-1}{p-1}},  \notag
\end{equation}%
where the operator $\mathcal{N}^{\prime }$\ is defined as in Lemma \ref{L1}.
\end{lemma}

\begin{proof}
Since $p>p_{0}>1,$ $p_{0}-1>0,$ $p-p_{0}>0,$\ we have that 
\begin{equation}
\frac{p-1}{p-p_{0}}=\frac{\left( p-p_{0}\right) +\left( p_{0}-1\right) }{%
p-p_{0}}=1+\frac{p_{0}-1}{p-p_{0}}>1.  \label{c16}
\end{equation}%
Then by applying the weighted H\"{o}lder inequality 
\begin{equation}
\sum\limits_{k=1}^{n}w(k)f(k)g(k)\leq \left[ \sum\limits_{k=1}^{n}w(k)f^{%
\gamma }(k)\right] ^{\frac{1}{\gamma }}\left[ \sum\limits_{k=1}^{n}w(k)g^{%
\nu }(k)\right] ^{\frac{1}{\nu }},  \label{wh}
\end{equation}%
on the term 
\begin{equation*}
\sum\limits_{k=1}^{n}w(k)\left( \mathcal{N}^{\prime }h\right) ^{\frac{p-p_{0}%
}{p-1}}(k),
\end{equation*}%
with indices $\gamma =\left( p-1\right) /\left( p-p_{0}\right) >1$ and $\nu
=\left( p-1\right) /\left( p_{0}-1\right) $ (where $1/\gamma +1/\nu =1$), we
see that%
\begin{eqnarray*}
&&\sum\limits_{k=1}^{n}w(k)\left( \mathcal{N}^{\prime }h\right) ^{\frac{%
p-p_{0}}{p-1}}(k)=\sum\limits_{k=1}^{n}w(k).1.\left( \mathcal{N}^{\prime
}h\right) ^{\frac{p-p_{0}}{p-1}}(k) \\
&\leq &\left( \sum\limits_{k=1}^{n}w(k)\right) ^{\frac{p_{0}-1}{p-1}}\left(
\sum\limits_{k=1}^{n}w(k)\mathcal{N}^{\prime }h(k)\right) ^{\frac{p-p_{0}}{%
p-1}},
\end{eqnarray*}%
and then 
\begin{eqnarray}
&&\frac{1}{n}\sum\limits_{k=1}^{n}w(k)\left( \mathcal{N}^{\prime }h\right) ^{%
\frac{p-p_{0}}{p-1}}(k)  \notag \\
&\leq &\left( \frac{1}{n}\sum\limits_{k=1}^{n}w(k)\right) ^{\frac{p_{0}-1}{%
p-1}}\left( \frac{1}{n}\sum\limits_{k=1}^{n}w(k)\mathcal{N}^{\prime
}h(k)\right) ^{\frac{p-p_{0}}{p-1}}.  \label{c17}
\end{eqnarray}%
Also, we note that%
\begin{align}
& \left( \frac{1}{n}\sum\limits_{k=1}^{n}\left( w(k)\left( \mathcal{N}%
^{\prime }h\right) ^{\frac{p-p_{0}}{p-1}}(k)\right) ^{\frac{-1}{p_{0}-1}%
}\right) ^{p_{0}-1}  \notag \\
& =\left( \frac{1}{n}\sum\limits_{k=1}^{n}w^{\frac{-1}{p_{0}-1}}(k)\left( 
\mathcal{N}^{\prime }h\right) ^{\frac{-\left( p-p_{0}\right) }{\left(
p_{0}-1\right) \left( p-1\right) }}(k)\right) ^{p_{0}-1}  \notag \\
& =\left( \frac{1}{n}\sum\limits_{k=1}^{n}\left( \mathcal{N}^{\prime
}h(k)\right) ^{\frac{-\left( p-p_{0}\right) }{\left( p_{0}-1\right) \left(
p-1\right) }}w^{\frac{-\left( p-p_{0}\right) }{\left( p_{0}-1\right) \left(
p-1\right) }}w^{\frac{p-p_{0}}{\left( p_{0}-1\right) \left( p-1\right) }%
}(k)w^{\frac{-1}{p_{0}-1}}(k)\right) ^{p_{0}-1}  \notag \\
& =\left( \frac{1}{n}\right) ^{p_{0}-1}\left( \sum\limits_{k=1}^{n}\left( 
\mathcal{N}^{\prime }h(k)\right) ^{\frac{-\left( p-p_{0}\right) }{\left(
p_{0}-1\right) \left( p-1\right) }}w^{\frac{-\left( p-p_{0}\right) }{\left(
p_{0}-1\right) \left( p-1\right) }}(k)w^{\frac{-1}{p-1}}(k)\right)
^{p_{0}-1}.  \label{c18}
\end{align}%
From\ the property $(c)$ in Lemma \ref{L1}, we observe that\ $\left( 
\mathcal{N}^{\prime }h\right) w\in \mathcal{A}_{1},$ and then 
\begin{equation}
\frac{1}{n}\sum\limits_{k=1}^{n}\left( \mathcal{N}^{\prime }h\right)
(k)w(k)\leq \left[ \left( \mathcal{N}^{\prime }h\right) w\right] _{\mathcal{A%
}_{1}}\left( \mathcal{N}^{\prime }h\right) (k)w(k).  \label{c19}
\end{equation}%
Substituting (\ref{c19}) into (\ref{c18}), we have for $p>p_{0}>1,$ $\left(
p-p_{0}\right) /\left[ \left( p_{0}-1\right) \left( p-1\right) \right] >0,$
that%
\begin{align}
& \left( \frac{1}{n}\sum\limits_{k=1}^{n}\left( w(k)\left( \mathcal{N}%
^{\prime }h\right) ^{\frac{p-p_{0}}{p-1}}(k)\right) ^{\frac{-1}{p_{0}-1}%
}\right) ^{p_{0}-1}  \notag \\
& \leq \left( \frac{1}{n}\sum\limits_{k=1}^{n}\left( \frac{1}{n}%
\sum\limits_{k=1}^{n}\left( \mathcal{N}^{\prime }h\right) (k)w(k)\right) ^{%
\frac{-\left( p-p_{0}\right) }{\left( p_{0}-1\right) \left( p-1\right) }}%
\left[ \left( \mathcal{N}^{\prime }h\right) w\right] _{\mathcal{A}_{1}}^{%
\frac{p-p_{0}}{\left( p_{0}-1\right) \left( p-1\right) }}w^{\frac{-1}{p-1}%
}(k)\right) ^{p_{0}-1}  \notag \\
& =\left[ \left( \mathcal{N}^{\prime }h\right) w\right] _{\mathcal{A}_{1}}^{%
\frac{p-p_{0}}{p-1}}\left( \frac{1}{n}\sum\limits_{k=1}^{n}\left( \mathcal{N}%
^{\prime }h\right) (k)w(k)\right) ^{\frac{-\left( p-p_{0}\right) }{p-1}%
}\left( \frac{1}{n}\sum\limits_{k=1}^{n}w^{\frac{-1}{p-1}}(k)\right)
^{p_{0}-1}.  \label{c20}
\end{align}%
From (\ref{c17}) and (\ref{c20}), we observe that%
\begin{eqnarray*}
&&\left( \frac{1}{n}\sum\limits_{k=1}^{n}w(k)\left( \mathcal{N}^{\prime
}h\right) ^{\frac{p-p_{0}}{p-1}}(k)\right) \left( \frac{1}{n}%
\sum\limits_{k=1}^{n}\left( w(k)\left( \mathcal{N}^{\prime }h\right) ^{\frac{%
p-p_{0}}{p-1}}(k)\right) ^{\frac{-1}{p_{0}-1}}\right) ^{p_{0}-1} \\
&\leq &\left[ \left( \mathcal{N}^{\prime }h\right) w\right] _{\mathcal{A}%
_{1}}^{\frac{p-p_{0}}{p-1}}\left[ \left( \frac{1}{n}\sum%
\limits_{k=1}^{n}w(k)\right) \left( \frac{1}{n}\sum\limits_{k=1}^{n}w^{\frac{%
-1}{p-1}}(k)\right) ^{p-1}\right] ^{\frac{p_{0}-1}{p-1}},
\end{eqnarray*}%
and then we have for $p>p_{0}>1$ and $w\in \mathcal{A}_{p},$ that%
\begin{equation*}
\left( \frac{1}{n}\sum\limits_{k=1}^{n}w(k)\left( \mathcal{N}^{\prime
}h\right) ^{\frac{p-p_{0}}{p-1}}(k)\right) \left( \frac{1}{n}%
\sum\limits_{k=1}^{n}\left( w(k)\left( \mathcal{N}^{\prime }h\right) ^{\frac{%
p-p_{0}}{p-1}}(k)\right) ^{\frac{-1}{p_{0}-1}}\right) ^{p_{0}-1}.
\end{equation*}%
Taking the supremum of the two sides of the last inequality$,$ we see that $%
w\left( \mathcal{N}^{\prime }h\right) ^{\frac{p-p_{0}}{p-1}}\in \mathcal{A}%
_{p_{0}}$ with the constant%
\begin{eqnarray*}
&&\left[ w\left( \mathcal{N}^{\prime }h\right) ^{\frac{p-p_{0}}{p-1}}\right]
_{\mathcal{A}_{p_{0}}}=\sup_{n\in \mathbb{Z}_{+}}\left( \frac{1}{n}%
\sum\limits_{k=1}^{n}w(k)\left( \mathcal{N}^{\prime }h\right) ^{\frac{p-p_{0}%
}{p-1}}(k)\right) \\
&&\times \left( \frac{1}{n}\sum\limits_{k=1}^{n}\left( w(k)\left( \mathcal{N}%
^{\prime }h\right) ^{\frac{p-p_{0}}{p-1}}(k)\right) ^{\frac{-1}{p_{0}-1}%
}\right) ^{p_{0}-1}\leq \left[ \left( \mathcal{N}^{\prime }h\right) w\right]
_{\mathcal{A}_{1}}^{\frac{p-p_{0}}{p-1}}\left[ w\right] _{\mathcal{A}_{p}}^{%
\frac{p_{0}-1}{p-1}},
\end{eqnarray*}%
which is the desired inequality. The proof is complete.
\end{proof}

\section{Main results}

In this section, we prove the main extrapolation theorem via\ the
Muckenhoupt weights with sharp constants. We assume that $f,$ $g$ are
nonnegative sequences belong to and $w_{0},$ $w$ are nonnegative weights.

\begin{theorem}
\label{T1}Assume that\ for some $p_{0},$ $1\leq p_{0}<\infty ,$ there exists
a positive increasing function $\mathcal{\varphi }_{p_{0}}$ on $[1,\infty )$
such that for every $w_{0}\in \mathcal{A}_{p_{0}},$%
\begin{equation}
\left( \sum_{k=1}^{\infty }w_{0}(k)f^{p_{0}}(k)\right) ^{\frac{1}{p_{0}}%
}\leq \mathcal{\varphi }_{p_{0}}\left( [w_{0}]_{\mathcal{A}_{p_{0}}}\right)
\left( \sum_{k=1}^{\infty }w_{0}(k)g^{p_{0}}(k)\right) ^{\frac{1}{p_{0}}}.
\label{*}
\end{equation}%
Then for all $p,$ $1<p<p_{0}<\infty $ and for all\ $w\in \mathcal{A}_{p},$%
\begin{equation}
\left( \sum_{k=1}^{\infty }w(k)f^{p}(k)\right) ^{\frac{1}{p}}\leq \mathcal{%
\varphi }_{p}\left( p_{0},p,[w]_{\mathcal{A}_{p}}\right) \left(
\sum_{k=1}^{\infty }w(k)g^{p}(k)\right) ^{\frac{1}{p}},  \label{*1}
\end{equation}%
where 
\begin{equation}
\mathcal{\varphi }_{p}\left( p_{0},p,[w]_{\mathcal{A}_{p}}\right) =2^{\frac{%
p_{0}-p}{p_{0}}}\mathcal{\varphi }_{p_{0}}\left( \left( 2\left\Vert \mathcal{%
M}\right\Vert _{\ell _{p}(w)}\right) ^{p_{0}-p}\left[ w\right] _{\mathcal{A}%
_{p}}\right) .  \label{**}
\end{equation}
\end{theorem}

\begin{proof}
Let $\varepsilon >0$ and define 
\begin{equation}
h(k):=\varepsilon \frac{f(k)}{\left\Vert f\right\Vert _{\ell _{p}(w)}}+\frac{%
g(k)}{\left\Vert g\right\Vert _{\ell _{p}(w)}},\text{ for \ }k\in \mathbb{Z}%
_{+}.  \label{c}
\end{equation}%
\ Applying the\ weighted Minkowski inequality, we have 
\begin{align}
\left\Vert h\right\Vert _{\ell _{p}(w)}& =\left( \sum_{k=1}^{\infty
}w(k)h^{p}(k)\right) ^{\frac{1}{p}}  \notag \\
& =\left( \sum_{k=1}^{\infty }w(k)\left( \varepsilon \frac{f(k)}{\left\Vert
f\right\Vert _{\ell _{p}(w)}}+\frac{g(k)}{\left\Vert g\right\Vert _{\ell
_{p}(w)}}\right) ^{p}\right) ^{\frac{1}{p}}  \notag \\
& \leq \frac{\varepsilon }{\left\Vert f\right\Vert _{\ell _{p}(w)}}\left(
\sum_{k=1}^{\infty }w(k)f^{p}(k)\right) ^{\frac{1}{p}}+\frac{1}{\left\Vert
g\right\Vert _{\ell _{p}(w)}}\left( \sum_{k=1}^{\infty }w(k)g^{p}(k)\right)
^{\frac{1}{p}}  \notag \\
& =\varepsilon +1.  \label{c1}
\end{align}%
Since%
\begin{eqnarray}
\left\Vert f\right\Vert _{\ell _{p}(w)} &=&\left( \sum_{k=1}^{\infty
}w(k)f^{p}(k)\right) ^{\frac{1}{p}}  \notag \\
&=&\left( \sum_{k=1}^{\infty }f^{p}(k)\mathcal{N}h^{\frac{-\left(
p_{0}-p\right) p}{p_{0}}}(k)\mathcal{N}h^{\frac{\left( p_{0}-p\right) p}{%
p_{0}}}(k)w(k)\right) ^{\frac{1}{p}},  \label{ab}
\end{eqnarray}%
we have by applying the weighted H\"{o}lder inequality on the\ right hand
side of (\ref{ab})\ with $\gamma =p_{0}/p>1$ and $\nu =p_{0}/\left(
p_{0}-p\right) $, that%
\begin{equation}
\left\Vert f\right\Vert _{\ell _{p}(w)}\leq \left( \sum_{k=1}^{\infty
}w(k)f^{p_{0}}(k)\left( \mathcal{N}h\right) ^{-\left( p_{0}-p\right)
}(k)\right) ^{\frac{1}{p_{0}}}\left( \sum_{k=1}^{\infty }w(k)\left( \mathcal{%
N}h\right) ^{p}(k)\right) ^{\frac{p_{0}-p}{pp_{0}}}.  \label{c2}
\end{equation}%
By applying\ the property\ $(ii)$ of Lemma \ref{L}, we have for $p_{0}>p,$\
that 
\begin{equation*}
\left( \sum_{k=1}^{\infty }w(k)\left( \mathcal{N}h\right) ^{p}(k)\right) ^{%
\frac{p_{0}-p}{pp_{0}}}=\left\Vert \mathcal{N}h\right\Vert _{\ell _{p}(w)}^{%
\frac{p_{0}-p}{p_{0}}}\leq \left( 2\left\Vert h\right\Vert _{\ell
_{p}(w)}\right) ^{\frac{p_{0}-p}{p_{0}}},
\end{equation*}%
and then by using (\ref{c1}), we obtain%
\begin{equation}
\left( \sum_{k=1}^{\infty }w(k)\left( \mathcal{N}h\right) ^{p}(k)\right) ^{%
\frac{p_{0}-p}{pp_{0}}}\leq \left[ 2\left( \varepsilon +1\right) \right] ^{%
\frac{p_{0}-p}{p_{0}}}.  \label{ac2}
\end{equation}%
Substituting (\ref{ac2}) into (\ref{c2}), we have that 
\begin{equation}
\left\Vert f\right\Vert _{\ell _{p}(w)}\leq \left[ 2\left( \varepsilon
+1\right) \right] ^{\frac{p_{0}-p}{p_{0}}}\left( \sum_{k=1}^{\infty
}w(k)f^{p_{0}}(k)\left( \mathcal{N}h\right) ^{-\left( p_{0}-p\right)
}(k)\right) ^{\frac{1}{p_{0}}}.  \label{c3}
\end{equation}%
From Lemma \ref{L*}, we have that $w\left( \mathcal{N}h\right) ^{-\left(
p_{0}-p\right) }\in \mathcal{A}_{p_{0}}$ with the constant 
\begin{equation}
\left[ w\left( \mathcal{N}h\right) ^{-\left( p_{0}-p\right) }\right] _{%
\mathcal{A}_{p_{0}}}\leq \left( \left[ \mathcal{N}h\right] _{\mathcal{A}%
_{1}}\right) ^{p_{0}-p}\left[ w\right] _{\mathcal{A}_{p}}.  \label{c4}
\end{equation}%
Since $w\left( \mathcal{N}h\right) ^{-\left( p_{0}-p\right) }\in \mathcal{A}%
_{p_{0}},$ we can\ apply (\ref{*}) with $w_{0}=w\mathcal{N}h^{-\left(
p_{0}-p\right) }\in \mathcal{A}_{p_{0}}$ and then%
\begin{eqnarray}
&&\left( \sum_{k=1}^{\infty }w(k)\left( \mathcal{N}h\right) ^{-\left(
p_{0}-p\right) }(k)f^{p_{0}}(k)\right) ^{\frac{1}{p_{0}}}  \notag \\
&\leq &\mathcal{\varphi }_{p_{0}}\left( [w\left( \mathcal{N}h\right)
^{-\left( p_{0}-p\right) }]_{\mathcal{A}_{p_{0}}}\right) \left(
\sum_{k=1}^{\infty }w(k)\left( \mathcal{N}h\right) ^{-\left( p_{0}-p\right)
}(k)g^{p_{0}}(k)\right) ^{\frac{1}{p_{0}}}.  \label{c5}
\end{eqnarray}%
Since the function $\mathcal{\varphi }_{p_{0}}$ is increasing, we have from (%
\ref{c4}) that%
\begin{equation*}
\mathcal{\varphi }_{p_{0}}\left( [w\left( \mathcal{N}h\right) ^{-\left(
p_{0}-p\right) }]_{\mathcal{A}_{p_{0}}}\right) \leq \mathcal{\varphi }%
_{p_{0}}\left( \left( \left[ \mathcal{N}h\right] _{\mathcal{A}_{1}}\right)
^{p_{0}-p}\left[ w\right] _{\mathcal{A}_{p}}\right) ,
\end{equation*}%
and then the inequality (\ref{c5}) becomes%
\begin{eqnarray}
&&\left( \sum_{k=1}^{\infty }w(k)\left( \mathcal{N}h\right) ^{-\left(
p_{0}-p\right) }(k)f^{p_{0}}(k)\right) ^{\frac{1}{p_{0}}}  \notag \\
&\leq &\mathcal{\varphi }_{p_{0}}\left( \left( \left[ \mathcal{N}h\right] _{%
\mathcal{A}_{1}}\right) ^{p_{0}-p}\left[ w\right] _{\mathcal{A}_{p}}\right)
\left( \sum_{k=1}^{\infty }w(k)\left( \mathcal{N}h\right) ^{-\left(
p_{0}-p\right) }(k)g^{p_{0}}(k)\right) ^{\frac{1}{p_{0}}}.  \label{c6}
\end{eqnarray}%
From\ the property\ $(iii)$ in\ Lemma \ref{L}, we observe that%
\begin{equation*}
\left[ \mathcal{N}h\right] _{\mathcal{A}_{1}}\leq 2\left\Vert \mathcal{M}%
\right\Vert _{\ell _{p}(w)},
\end{equation*}%
and then we have for\ $p_{0}>p$ and the\ increasing function $\mathcal{%
\varphi }_{p_{0}}$ that%
\begin{equation}
\mathcal{\varphi }_{p_{0}}\left( \left( \left[ \mathcal{N}h\right] _{%
\mathcal{A}_{1}}\right) ^{p_{0}-p}\left[ w\right] _{\mathcal{A}_{p}}\right)
\leq \mathcal{\varphi }_{p_{0}}\left( \left( 2\left\Vert \mathcal{M}%
\right\Vert _{\ell _{p}(w)}\right) ^{p_{0}-p}\left[ w\right] _{\mathcal{A}%
_{p}}\right) .  \label{c7}
\end{equation}%
Substituting (\ref{c7}) into (\ref{c6}), we obtain%
\begin{eqnarray}
&&\left( \sum_{k=1}^{\infty }w(k)\left( \mathcal{N}h\right) ^{-\left(
p_{0}-p\right) }(k)f^{p_{0}}(k)\right) ^{\frac{1}{p_{0}}}  \notag \\
&\leq &\mathcal{\varphi }_{p_{0}}\left( \left( 2\left\Vert \mathcal{M}%
\right\Vert _{\ell _{p}(w)}\right) ^{p_{0}-p}\left[ w\right] _{\mathcal{A}%
_{p}}\right) \left( \sum_{k=1}^{\infty }w(k)\left( \mathcal{N}h\right)
^{-\left( p_{0}-p\right) }(k)g^{p_{0}}(k)\right) ^{\frac{1}{p_{0}}}.
\label{c8}
\end{eqnarray}%
Substituting (\ref{c8}) into (\ref{c3}), we get%
\begin{eqnarray}
\left\Vert f\right\Vert _{\ell _{p}(w)} &\leq &\left[ 2\left( \varepsilon
+1\right) \right] _{p_{0}}^{\frac{p_{0}-p}{p_{0}}}\mathcal{\varphi }%
_{p_{0}}\left( \left( 2\left\Vert \mathcal{M}\right\Vert _{\ell
_{p}(w)}\right) ^{p_{0}-p}\left[ w\right] _{\mathcal{A}_{p}}\right)  \notag
\\
&&\times \left( \sum_{k=1}^{\infty }w(k)\left( \mathcal{N}h\right) ^{-\left(
p_{0}-p\right) }(k)g^{p_{0}}(k)\right) ^{\frac{1}{p_{0}}}.  \label{c9}
\end{eqnarray}%
From (\ref{c}), where $\varepsilon >0$ and\ the property\ $(i)$ in\ Lemma %
\ref{L}, we observe that%
\begin{equation}
\frac{g(k)}{\left\Vert g\right\Vert _{\ell _{p}(w)}}\leq h(k)\leq \mathcal{N}%
h(k).  \label{c10}
\end{equation}%
Substituting (\ref{c10}) into (\ref{c9}), we have for $p_{0}>p,$\ that%
\begin{eqnarray*}
\left\Vert f\right\Vert _{\ell _{p}(w)} &\leq &\left[ 2\left( \varepsilon
+1\right) \right] _{p_{0}}^{\frac{p_{0}-p}{p_{0}}}\mathcal{\varphi }%
_{p_{0}}\left( \left( 2\left\Vert \mathcal{M}\right\Vert _{\ell
_{p}(w)}\right) ^{p_{0}-p}\left[ w\right] _{\mathcal{A}_{p}}\right) \\
&&\times \left\Vert g\right\Vert _{\ell _{p}(w)}^{\frac{p_{0}-p}{p_{0}}%
}\left( \sum_{k=1}^{\infty }w(k)\left[ g(k)\right] ^{-\left( p_{0}-p\right)
}g^{p_{0}}(k)\right) ^{\frac{1}{p_{0}}} \\
&=&\left[ 2\left( \varepsilon +1\right) \right] ^{\frac{p_{0}-p}{p_{0}}}%
\mathcal{\varphi }_{p_{0}}\left( \left( 2\left\Vert \mathcal{M}\right\Vert
_{\ell _{p}(w)}\right) ^{p_{0}-p}\left[ w\right] _{\mathcal{A}_{p}}\right) \\
&&\times \left\Vert g\right\Vert _{\ell _{p}(w)}^{\frac{p_{0}-p}{p_{0}}%
}\left( \sum_{k=1}^{\infty }w(k)g^{p}(k)\right) ^{\frac{1}{p_{0}}} \\
&=&\left[ 2\left( \varepsilon +1\right) \right] ^{\frac{p_{0}-p}{p_{0}}}%
\mathcal{\varphi }_{p_{0}}\left( \left( 2\left\Vert \mathcal{M}\right\Vert
_{\ell _{p}(w)}\right) ^{p_{0}-p}\left[ w\right] _{\mathcal{A}_{p}}\right)
\left\Vert g\right\Vert _{\ell _{p}(w)},
\end{eqnarray*}%
i.e.%
\begin{eqnarray*}
&&\left( \sum_{k=1}^{\infty }w(k)f^{p}(k)\right) ^{\frac{1}{p}} \\
&\leq &\left[ 2\left( \varepsilon +1\right) \right] ^{\frac{p_{0}-p}{p_{0}}}%
\mathcal{\varphi }_{p_{0}}\left( \left( 2\left\Vert \mathcal{M}\right\Vert
_{\ell _{p}(w)}\right) ^{p_{0}-p}\left[ w\right] _{\mathcal{A}_{p}}\right)
\left( \sum_{k=1}^{\infty }w(k)g^{p}(k)\right) ^{\frac{1}{p}}.
\end{eqnarray*}%
Taking the limit when $\varepsilon \rightarrow 0,$ we have that%
\begin{eqnarray*}
&&\left( \sum_{k=1}^{\infty }w(k)f^{p}(k)\right) ^{\frac{1}{p}} \\
&\leq &2^{\frac{p_{0}-p}{p_{0}}}\mathcal{\varphi }_{p_{0}}\left( \left(
2\left\Vert \mathcal{M}\right\Vert _{\ell _{p}(w)}\right) ^{p_{0}-p}\left[ w%
\right] _{\mathcal{A}_{p}}\right) \left( \sum_{k=1}^{\infty
}w(k)g^{p}(k)\right) ^{\frac{1}{p}},
\end{eqnarray*}%
which satisfies (\ref{*1}).\ The proof is complete.
\end{proof}

\begin{theorem}
\label{T2}Assume that\ for some $p_{0},$ $1\leq p_{0}<\infty ,$ there exists
a positive increasing function $\mathcal{\varphi }_{p_{0}}$ on $[1,\infty )$
such that for every $w_{0}\in \mathcal{A}_{p_{0}},$%
\begin{equation}
\left( \sum_{k=1}^{\infty }w_{0}(k)f^{p_{0}}(k)\right) ^{\frac{1}{p_{0}}%
}\leq \mathcal{\varphi }_{p_{0}}\left( [w_{0}]_{\mathcal{A}_{p_{0}}}\right)
\left( \sum_{k=1}^{\infty }w_{0}(k)g^{p_{0}}(k)\right) ^{\frac{1}{p_{0}}}.
\label{*2}
\end{equation}%
Then for all $p,$ $1<p_{0}<p<\infty $\ and for all\ $w\in \mathcal{A}_{p},$%
\begin{equation}
\left( \sum_{k=1}^{\infty }w(k)f^{p}(k)\right) ^{\frac{1}{p}}\leq \mathcal{%
\varphi }_{p}\left( p_{0},p,[w]_{\mathcal{A}_{p}}\right) \left(
\sum_{k=1}^{\infty }w(k)g^{p}(k)\right) ^{\frac{1}{p}},  \label{*12}
\end{equation}%
where 
\begin{equation}
\mathcal{\varphi }_{p}\left( p_{0},p,[w]_{\mathcal{A}_{p}}\right) =2^{\frac{%
p-p_{0}}{\left( p-1\right) p_{0}}}\mathcal{\varphi }_{p_{0}}\left( \left(
2\left\Vert \mathcal{M}\right\Vert _{\ell _{p^{\prime }}(w^{1-p^{\prime
}})}\right) ^{\frac{p-p_{0}}{p-1}}\left[ w\right] _{\mathcal{A}_{p}}^{\frac{%
p_{0}-1}{p-1}}\right) ,  \label{**12}
\end{equation}%
with $p^{\prime }=p/\left( p-1\right) .$
\end{theorem}

\begin{proof}
Since $f\in \ell _{p}(w),$ we may assume without loss of generality that
there exists a nonnegative sequence\ $h\in \ell _{p^{\prime }}(w)$, $%
\left\Vert h\right\Vert _{\ell _{p^{\prime }}(w)}=1,$ such that%
\begin{equation}
\left\Vert f\right\Vert _{\ell _{p}(w)}=\sum_{k=1}^{\infty }f(k)h(k)w(k).
\label{c12}
\end{equation}%
By applying the property $(a)$\ in\ Lemma \ref{L1} on the right hand side\
of (\ref{c12}), we have for $p>p_{0}>1$\ that%
\begin{eqnarray}
\sum_{k=1}^{\infty }f(k)h(k)w(k) &=&\sum_{k=1}^{\infty }f(k)h^{\frac{p-p_{0}%
}{p_{0}\left( p-1\right) }}(k)h^{\frac{p\left( p_{0}-1\right) }{p_{0}\left(
p-1\right) }}(k)w(k)  \notag \\
&\leq &\sum_{k=1}^{\infty }f(k)\left( \mathcal{N}^{\prime }h(k)\right) ^{%
\frac{p-p_{0}}{p_{0}\left( p-1\right) }}h^{\frac{p\left( p_{0}-1\right) }{%
p_{0}\left( p-1\right) }}(k)w(k).  \label{c13}
\end{eqnarray}%
Applying the weighted H\"{o}lder inequality on the term%
\begin{equation*}
\sum_{k=1}^{\infty }f(k)\left( \mathcal{N}^{\prime }h(k)\right) ^{\frac{%
p-p_{0}}{p_{0}\left( p-1\right) }}h^{\frac{p\left( p_{0}-1\right) }{%
p_{0}\left( p-1\right) }}(k)w(k),
\end{equation*}%
with indices $p_{0}>1$ and $p_{0}=p_{0}/\left( p_{0}-1\right) ,$ we see that%
\begin{eqnarray*}
&&\sum_{k=1}^{\infty }f(k)\left( \mathcal{N}^{\prime }h(k)\right) ^{\frac{%
p-p_{0}}{p_{0}\left( p-1\right) }}h^{\frac{p\left( p_{0}-1\right) }{%
p_{0}\left( p-1\right) }}(k)w(k) \\
&\leq &\left( \sum_{k=1}^{\infty }f^{p_{0}}(k)\left( \mathcal{N}^{\prime
}h(k)\right) ^{\frac{p-p_{0}}{p-1}}w(k)\right) ^{\frac{1}{p_{0}}}\left(
\sum_{k=1}^{\infty }h^{\frac{p}{p-1}}(k)w(k)\right) ^{\frac{p_{0}-1}{p_{0}}},
\end{eqnarray*}%
and then since $\left\Vert h\right\Vert _{\ell _{p^{\prime }}(w)}=1$, we that%
\begin{equation}
\sum_{k=1}^{\infty }f(k)\left( \mathcal{N}^{\prime }h\right) ^{\frac{p-p_{0}%
}{p_{0}\left( p-1\right) }}(k)h^{\frac{p\left( p_{0}-1\right) }{p_{0}\left(
p-1\right) }}(k)w(k)\leq \left( \sum_{k=1}^{\infty }f^{p_{0}}(k)\left( 
\mathcal{N}^{\prime }h\right) ^{\frac{p-p_{0}}{p-1}}(k)w(k)\right) ^{\frac{1%
}{p_{0}}}.  \label{c14}
\end{equation}%
Substituting (\ref{c14}) into (\ref{c13}), we observe that%
\begin{equation}
\sum_{k=1}^{\infty }f(k)h(k)w(k)\leq \left( \sum_{k=1}^{\infty
}f^{p_{0}}(k)\left( \mathcal{N}^{\prime }h\right) ^{\frac{p-p_{0}}{p-1}%
}(k)w(k)\right) ^{\frac{1}{p_{0}}},  \notag
\end{equation}%
and then we have from (\ref{c12}) that%
\begin{equation}
\left\Vert f\right\Vert _{\ell _{p}(w)}\leq \left( \sum_{k=1}^{\infty
}f^{p_{0}}(k)w(k)\left( \mathcal{N}^{\prime }h\right) ^{\frac{p-p_{0}}{p-1}%
}(k)\right) ^{\frac{1}{p_{0}}}.  \label{c15}
\end{equation}%
From Lemma \ref{L1*}, we have that\ $w\left( \mathcal{N}^{\prime }h\right) ^{%
\frac{p-p_{0}}{p-1}}\in \mathcal{A}_{p_{0}}$ with the constant%
\begin{equation}
\left[ w\left( \mathcal{N}^{\prime }h\right) ^{\frac{p-p_{0}}{p-1}}\right] _{%
\mathcal{A}_{p_{0}}}\leq \left[ \left( \mathcal{N}^{\prime }h\right) w\right]
_{\mathcal{A}_{1}}^{\frac{p-p_{0}}{p-1}}\left[ w\right] _{\mathcal{A}_{p}}^{%
\frac{p_{0}-1}{p-1}},  \label{c21}
\end{equation}%
Since $w\left( \mathcal{N}^{\prime }h\right) ^{\frac{p-p_{0}}{p-1}}\in 
\mathcal{A}_{p_{0}},$ we can apply (\ref{*2}) with $w_{0}=w\left( \mathcal{N}%
^{\prime }h\right) ^{\frac{p-p_{0}}{p-1}}\in \mathcal{A}_{p_{0}}$ to get%
\begin{eqnarray}
&&\left( \sum_{k=1}^{\infty }w(k)\left( \mathcal{N}^{\prime }h\right) ^{%
\frac{p-p_{0}}{p-1}}(k)f^{p_{0}}(k)\right) ^{\frac{1}{p_{0}}}  \notag \\
&\leq &\mathcal{\varphi }_{p_{0}}\left( [w\left( \mathcal{N}^{\prime
}h\right) ^{\frac{p-p_{0}}{p-1}}]_{\mathcal{A}_{p_{0}}}\right) \left(
\sum_{k=1}^{\infty }w(k)\left( \mathcal{N}^{\prime }h\right) ^{\frac{p-p_{0}%
}{p-1}}(k)g^{p_{0}}(k)\right) ^{\frac{1}{p_{0}}}.  \label{c22}
\end{eqnarray}%
Since $\mathcal{\varphi }_{p_{0}}$ is an increasing function, we have from (%
\ref{c21}) that%
\begin{equation}
\mathcal{\varphi }_{p_{0}}\left( \left[ w\left( \mathcal{N}^{\prime
}h\right) ^{\frac{p-p_{0}}{p-1}}\right] _{\mathcal{A}_{p_{0}}}\right) \leq 
\mathcal{\varphi }_{p_{0}}\left( \left[ \left( \mathcal{N}^{\prime }h\right)
w\right] _{\mathcal{A}_{1}}^{\frac{p-p_{0}}{p-1}}\left[ w\right] _{\mathcal{A%
}_{p}}^{\frac{p_{0}-1}{p-1}}\right) ,  \notag
\end{equation}%
and then the inequality (\ref{c22}) becomes 
\begin{eqnarray}
&&\left( \sum_{k=1}^{\infty }w(k)\left( \mathcal{N}^{\prime }h\right) ^{%
\frac{p-p_{0}}{p-1}}(k)f^{p_{0}}(k)\right) ^{\frac{1}{p_{0}}}  \notag \\
&\leq &\mathcal{\varphi }_{p_{0}}\left( \left[ \left( \mathcal{N}^{\prime
}h\right) w\right] _{\mathcal{A}_{1}}^{\frac{p-p_{0}}{p-1}}\left[ w\right] _{%
\mathcal{A}_{p}}^{\frac{p_{0}-1}{p-1}}\right) \left( \sum_{k=1}^{\infty
}w(k)\left( \mathcal{N}^{\prime }h\right) ^{\frac{p-p_{0}}{p-1}%
}(k)g^{p_{0}}(k)\right) ^{\frac{1}{p_{0}}}.  \label{c23}
\end{eqnarray}%
Again by applying the property $(c)$ in Lemma \ref{L1} , we have for\ $%
p>p_{0}>1$\ that\ 
\begin{equation*}
\left[ \left( \mathcal{N}^{\prime }h\right) w\right] _{\mathcal{A}_{1}}^{%
\frac{p-p_{0}}{p-1}}\left[ w\right] _{\mathcal{A}_{p}}^{\frac{p_{0}-1}{p-1}%
}\leq \left( 2\left\Vert \mathcal{M}^{\prime }\right\Vert _{\ell _{p^{\prime
}}(w)}\right) ^{\frac{p-p_{0}}{p-1}}\left[ w\right] _{\mathcal{A}_{p}}^{%
\frac{p_{0}-1}{p-1}}.
\end{equation*}%
Since $\mathcal{\varphi }_{p_{0}}$ is an increasing function, we get%
\begin{equation}
\mathcal{\varphi }_{p_{0}}\left( \left[ \left( \mathcal{N}^{\prime }h\right)
w\right] _{\mathcal{A}_{1}}^{\frac{p-p_{0}}{p-1}}\left[ w\right] _{\mathcal{A%
}_{p}}^{\frac{p_{0}-1}{p-1}}\right) \leq \mathcal{\varphi }_{p_{0}}\left(
\left( 2\left\Vert \mathcal{M}^{\prime }\right\Vert _{\ell _{p^{\prime
}}(w)}\right) ^{\frac{p-p_{0}}{p-1}}\left[ w\right] _{\mathcal{A}_{p}}^{%
\frac{p_{0}-1}{p-1}}\right) .  \label{c24}
\end{equation}%
Since $\mathcal{M}^{\prime }f(k)=\left( 1/w(k)\right) \mathcal{M}\left(
fw\right) (k),$ we have that%
\begin{eqnarray}
\left\Vert \mathcal{M}^{\prime }\right\Vert _{\ell _{p^{\prime }}(w)}
&=&\left( \sum_{k=1}^{\infty }w(k)\left( \mathcal{M}^{\prime }f(k)\right)
^{p^{\prime }}\right) ^{\frac{1}{p^{\prime }}}  \notag \\
&=&\left( \sum_{k=1}^{\infty }w^{1-p^{\prime }}(k)\mathcal{M}^{p^{\prime
}}\left( fw\right) (k)\right) ^{\frac{1}{p^{\prime }}}=\left\Vert \mathcal{M}%
\right\Vert _{\ell _{p^{\prime }}(w^{1-p^{\prime }})}.  \label{c25}
\end{eqnarray}%
Substituting (\ref{c25}) into (\ref{c24}), we get%
\begin{equation*}
\mathcal{\varphi }_{p_{0}}\left( \left[ \left( \mathcal{N}^{\prime }h\right)
w\right] _{\mathcal{A}_{1}}^{\frac{p-p_{0}}{p-1}}\left[ w\right] _{\mathcal{A%
}_{p}}^{\frac{p_{0}-1}{p-1}}\right) \leq \mathcal{\varphi }_{p_{0}}\left(
\left( 2\left\Vert \mathcal{M}\right\Vert _{\ell _{p^{\prime
}}(w^{1-p^{\prime }})}\right) ^{\frac{p-p_{0}}{p-1}}\left[ w\right] _{%
\mathcal{A}_{p}}^{\frac{p_{0}-1}{p-1}}\right) ,
\end{equation*}%
and then the inequality (\ref{c23}) becomes%
\begin{eqnarray}
\left( \sum_{k=1}^{\infty }w(k)\left( \mathcal{N}^{\prime }h\right) ^{\frac{%
p-p_{0}}{p-1}}(k)f^{p_{0}}(k)\right) ^{\frac{1}{p_{0}}} &\leq &\mathcal{%
\varphi }_{p_{0}}\left( \left( 2\left\Vert \mathcal{M}\right\Vert _{\ell
_{p^{\prime }}(w^{1-p^{\prime }})}\right) ^{\frac{p-p_{0}}{p-1}}\left[ w%
\right] _{\mathcal{A}_{p}}^{\frac{p_{0}-1}{p-1}}\right)  \notag \\
&&\times \left( \sum_{k=1}^{\infty }w(k)\left( \mathcal{N}^{\prime }h\right)
^{\frac{p-p_{0}}{p-1}}(k)g^{p_{0}}(k)\right) ^{\frac{1}{p_{0}}}.  \label{c26}
\end{eqnarray}%
From\ (\ref{c15}) and (\ref{c26}), we have that%
\begin{eqnarray}
\left\Vert f\right\Vert _{\ell _{p}(w)} &\leq &\mathcal{\varphi }%
_{p_{0}}\left( \left( 2\left\Vert \mathcal{M}\right\Vert _{\ell ^{p^{\prime
}}(w^{1-p^{\prime }})}\right) ^{\frac{p-p_{0}}{p-1}}\left[ w\right] _{%
\mathcal{A}_{p}}^{\frac{p_{0}-1}{p-1}}\right)  \notag \\
&&\times \left( \sum_{k=1}^{\infty }w(k)\left( \mathcal{N}^{\prime }h\right)
^{\frac{p-p_{0}}{p-1}}(k)g^{p_{0}}(k)\right) ^{\frac{1}{p_{0}}}.  \label{c27}
\end{eqnarray}%
Applying the weighted H\"{o}lder inequality on the term%
\begin{equation*}
\sum_{k=1}^{\infty }w(k)\left( \mathcal{N}^{\prime }h\right) ^{\frac{p-p_{0}%
}{p-1}}(k)g^{p_{0}}(k),
\end{equation*}%
with $\gamma =p/p_{0}>1$ and $\nu =p/\left( p-p_{0}\right) ,$ we see that%
\begin{eqnarray}
&&\sum_{k=1}^{\infty }w(k)\left( \mathcal{N}^{\prime }h\right) ^{\frac{%
p-p_{0}}{p-1}}(k)g^{p_{0}}(k)  \notag \\
&\leq &\left( \sum_{k=1}^{\infty }w(k)\left( \mathcal{N}^{\prime }h\right) ^{%
\frac{p}{p-1}}(k)\right) ^{\frac{p-p_{0}}{p}}\left( \sum_{k=1}^{\infty
}w(k)g^{p}(k)\right) ^{\frac{p_{0}}{p}}.  \label{c28}
\end{eqnarray}%
Substituting (\ref{c28}) into (\ref{c27}), we obtain%
\begin{eqnarray}
\left\Vert f\right\Vert _{\ell _{p}(w)} &\leq &\mathcal{\varphi }%
_{p_{0}}\left( \left( 2\left\Vert \mathcal{M}\right\Vert _{\ell _{p^{\prime
}}(w^{1-p^{\prime }})}\right) ^{\frac{p-p_{0}}{p-1}}\left[ w\right] _{%
\mathcal{A}_{p}}^{\frac{p_{0}-1}{p-1}}\right)  \notag \\
&&\times \left( \sum_{k=1}^{\infty }w(k)\left( \mathcal{N}^{\prime }h\right)
^{\frac{p}{p-1}}(k)\right) ^{\frac{p-p_{0}}{pp_{0}}}\left(
\sum_{k=1}^{\infty }w(k)g^{p}(k)\right) ^{\frac{1}{p}}.  \label{c29}
\end{eqnarray}%
From the property $(b)$ in Lemma \ref{L1}, we have for $p>p_{0}>1$ that 
\begin{equation*}
\left( \sum_{k=1}^{\infty }w(k)\left( \mathcal{N}^{\prime }h\right) ^{\frac{p%
}{p-1}}(k)\right) ^{\frac{p-p_{0}}{pp_{0}}}=\left\Vert \mathcal{N}^{\prime
}h\right\Vert _{\ell _{p^{\prime }}(w)}^{\frac{p-p_{0}}{\left( p-1\right)
p_{0}}}\leq \left( 2\left\Vert h\right\Vert _{\ell _{p^{\prime }}(w)}\right)
^{\frac{p-p_{0}}{\left( p-1\right) p_{0}}}.
\end{equation*}%
Since $\left\Vert h\right\Vert _{\ell _{p^{\prime }}(w)}=1$, we have that%
\begin{equation}
\left( \sum_{k=1}^{\infty }w(k)\left( \mathcal{N}^{\prime }h\right) ^{\frac{p%
}{p-1}}(k)\right) ^{\frac{p-p_{0}}{pp_{0}}}\leq 2^{\frac{p-p_{0}}{\left(
p-1\right) p_{0}}}.  \label{c30}
\end{equation}%
Substituting (\ref{c30}) into (\ref{c29}), we obtain%
\begin{eqnarray*}
\left\Vert f\right\Vert _{\ell _{p}(w)} &\leq &2^{\frac{p-p_{0}}{\left(
p-1\right) p_{0}}}\mathcal{\varphi }_{p_{0}}\left( \left( 2\left\Vert 
\mathcal{M}\right\Vert _{\ell ^{p^{\prime }}(w^{1-p^{\prime }})}\right) ^{%
\frac{p-p_{0}}{p-1}}\left[ w\right] _{\mathcal{A}_{p}}^{\frac{p_{0}-1}{p-1}%
}\right) \\
&&\times \left( \sum_{k=1}^{\infty }w(k)g^{p}(k)\right) ^{\frac{1}{p}},
\end{eqnarray*}%
i.e.%
\begin{eqnarray*}
\left( \sum_{k=1}^{\infty }w(k)f^{p}(k)\right) ^{\frac{1}{p}} &\leq &2^{%
\frac{p-p_{0}}{\left( p-1\right) p_{0}}}\mathcal{\varphi }_{p_{0}}\left(
\left( 2\left\Vert \mathcal{M}\right\Vert _{\ell ^{p^{\prime
}}(w^{1-p^{\prime }})}\right) ^{\frac{p-p_{0}}{p-1}}\left[ w\right] _{%
\mathcal{A}_{p}}^{\frac{p_{0}-1}{p-1}}\right) \\
&&\times \left( \sum_{k=1}^{\infty }w(k)g^{p}(k)\right) ^{\frac{1}{p}},
\end{eqnarray*}%
which satisfies (\ref{*12}). The proof is complete.
\end{proof}

By combining Theorem \ref{T1} and\ Theorem \ref{T2}, we get the following
theorem.

\begin{theorem}
\label{T3}Assume that\ for some $p_{0},$ $1\leq p_{0}<\infty ,$ there exists
a positive increasing function $\mathcal{\varphi }_{p_{0}}$ on $[1,\infty )$
such that for every $w_{0}\in \mathcal{A}_{p_{0}},$%
\begin{equation}
\sum_{k=1}^{\infty }w_{0}(k)f^{p_{0}}(k)\leq \mathcal{\varphi }%
_{p_{0}}\left( [w_{0}]_{\mathcal{A}_{p_{0}}}\right) \sum_{k=1}^{\infty
}w_{0}(k)g^{p_{0}}(k).  \label{3,2}
\end{equation}%
Then for all $p,$ $1\leq p<\infty $ and for all\ $w\in \mathcal{A}_{p},$%
\begin{equation*}
\sum_{k=1}^{\infty }w(k)f^{p}(k)\leq \mathcal{\varphi }_{p}\left(
p_{0},p,[w]_{\mathcal{A}_{p}}\right) \sum_{k=1}^{\infty }w(k)g^{p}(k).
\end{equation*}%
If $p<p_{0},$ then 
\begin{equation*}
\mathcal{\varphi }_{p}\left( p_{0},p,[w]_{\mathcal{A}_{p}}\right) =2^{\frac{%
p_{0}-p}{p_{0}}}\mathcal{\varphi }_{p_{0}}\left( \left( 2\left\Vert \mathcal{%
M}\right\Vert _{\ell _{p}(w)}\right) ^{p_{0}-p}\left[ w\right] _{\mathcal{A}%
_{p}}\right) ,
\end{equation*}%
and if $p>p_{0},$ then 
\begin{equation*}
\mathcal{\varphi }_{p}\left( p_{0},p,[w]_{\mathcal{A}_{p}}\right) =2^{\frac{%
p-p_{0}}{\left( p-1\right) p_{0}}}\mathcal{\varphi }_{p_{0}}\left( \left(
2\left\Vert \mathcal{M}\right\Vert _{\ell _{(w^{1-p^{\prime }})}^{p^{\prime
}}}\right) ^{\frac{p-p_{0}}{p-1}}\left[ w\right] _{\mathcal{A}_{p}}^{\frac{%
p_{0}-1}{p-1}}\right) ,
\end{equation*}%
with $p^{\prime }=p/\left( p-1\right) .$
\end{theorem}

Now, we are ready to state and prove the main theorem.

\begin{theorem}
Assume that $\mathcal{T}$ be a given operator defined on $\mathbb{Z}_{+}$
and there exists a positive an increasing function $\mathcal{\varphi }%
_{p_{0}}$ on $[1,\infty )$ such that for every $w_{0}\in \mathcal{A}_{p_{0}}$%
\begin{equation*}
\sum_{k=1}^{\infty }\left\vert \mathcal{T}f(k)\right\vert
^{p_{0}}w_{0}(k)\leq \varphi _{p_{0}}\left( \left[ w\right] _{\mathcal{A}%
_{p_{0}}}\right) \sum_{k=1}^{\infty }\left\vert f(k)\right\vert
^{p_{0}}w_{0}(k),
\end{equation*}%
\ for some $p_{0},$ $1\leq p_{0}<\infty .$ Then 
\begin{equation*}
\sum_{k=1}^{\infty }\left\vert \mathcal{T}f(k)\right\vert ^{p_{0}}w(k)\leq 
\mathcal{\varphi }_{p}\left( p_{0},p,[w]_{\mathcal{A}_{p}}\right)
\sum_{k=1}^{\infty }\left\vert f(k)\right\vert ^{p_{0}}w(k),
\end{equation*}%
for every $w\in \mathcal{A}_{p}$ such that $1\leq p<\infty $.

If $p<p_{0},$ then 
\begin{equation*}
\mathcal{\varphi }_{p}\left( p_{0},p,[w]_{\mathcal{A}_{p}}\right) =2^{\frac{%
p_{0}-p}{p_{0}}}\mathcal{\varphi }_{p_{0}}\left( \left( 2\left\Vert \mathcal{%
M}\right\Vert _{\ell _{p}(w)}\right) ^{p_{0}-p}\left[ w\right] _{\mathcal{A}%
_{p}}\right) .
\end{equation*}%
If $p>p_{0},$ then 
\begin{equation*}
\mathcal{\varphi }_{p}\left( p_{0},p,[w]_{\mathcal{A}_{p}}\right) =2^{\frac{%
p-p_{0}}{\left( p-1\right) p_{0}}}\mathcal{\varphi }_{p_{0}}\left( \left(
2\left\Vert \mathcal{M}\right\Vert _{\ell _{p^{\prime }}(w^{1-p^{\prime
}})}\right) ^{\frac{p-p_{0}}{p-1}}\left[ w\right] _{\mathcal{A}_{p}}^{\frac{%
p_{0}-1}{p-1}}\right) ,
\end{equation*}%
where $p^{\prime }=p/\left( p-1\right) .$
\end{theorem}

\begin{proof}
The proof is obtained from Theorem \ref{T3} by considering the pair $(%
\mathcal{T}f,f)$ instead of $(f,g).$ The proof is complete.
\end{proof}

In the following, we prove two corollaries of Theorem \ref{T3}. These
results further illustrate the value of the extrapolation of families of
pairs of functions. To prove these results we need the following inclusion
properties of Muckenhoupt classes that has been proved in \cite{samir}.

\begin{theorem}
\label{Muck2}Let $w$ be a positive weight, $p,$ $q$ be nonnegative real
numbers. Then the following inclusion relations hold:

$(1)$. $\mathcal{A}_{1}\subset \mathcal{A}_{p}\subset \mathcal{A}_{\infty }$%
, for all $1<p<\infty ,$

$(2)$. $\mathcal{A}_{p}\subset \mathcal{A}_{q}$ for all $1<p\leq q,$

$(3).$ $\mathcal{A}_{\infty }=\bigcup\limits_{1<p}\mathcal{A}_{p}$ with $%
\left[ w\right] _{\mathcal{A}_{\infty }}=\lim_{p\rightarrow \infty }\left[ w%
\right] _{\mathcal{A}_{p}}$ and $\mathcal{A}_{1}\subset \bigcap {}_{p>1}%
\mathcal{A}_{p}.$
\end{theorem}

Our first theorem yields $\ell _{p}$ inequalities with weights in $\mathcal{A%
}_{p/r}$, $p\geq r>0.$ Hereafter, by $\Omega $ we mean a family of pairs $%
(f\,,g)$ of nonnegative sequences.

\begin{corollary}
\label{Cv1}Suppose that\ for some $r>0$, and some $p_{0}$, $r\leq
p_{0}<\infty ,$ and for every $w_{0}\in \mathcal{A}_{p_{0}/r},$%
\begin{equation}
\sum_{k=1}^{\infty }w_{0}(k)f^{p_{0}}(k)\leq C\sum_{k=1}^{\infty
}w_{0}(k)g^{p_{0}}(k)\text{.}  \label{as}
\end{equation}%
Then for all $p,$ $r\leq p<\infty $ and all\ $w\in \mathcal{A}_{p/r},$%
\begin{equation}
\sum_{k=1}^{\infty }w(k)f^{p}(k)\leq C\sum_{k=1}^{\infty }w(k)g^{p}(k).
\label{as1}
\end{equation}
\end{corollary}

\begin{proof}
Define a new family $\Omega _{r}$ consisting of the pairs $(f^{r},g^{r})$,
where $(f,g)\in \Omega .$ Then for $w_{0}\in \mathcal{A}_{p_{0}/r}$, by (\ref%
{as}) 
\begin{eqnarray*}
\sum_{k=1}^{\infty }w_{0}(k)\left( f^{r}(k)\right) ^{p_{0}/r}
&=&\sum_{k=1}^{\infty }w_{0}(k)f^{p_{0}}(k)\leq C\sum_{k=1}^{\infty
}w_{0}(k)g^{p_{0}}(k) \\
&=&C\sum_{k=1}^{\infty }w_{0}(k)\left( g^{r}(k)\right) ^{p_{0}/r}.
\end{eqnarray*}%
Therefore, we can apply Theorem \ref{T3} with this as our initial hypothesis
and conclude that for all $q>1$ and all $w\in \mathcal{A}_{q},$%
\begin{equation*}
\sum_{k=1}^{\infty }w(k)\left( f^{r}(k)\right) ^{q}\leq C\sum_{k=1}^{\infty
}w(k)\left( g^{r}(k)\right) ^{q}\text{.}
\end{equation*}%
By putting $q=p/r$ for some $p>r$ and this is equivalent to (\ref{as1}). The
proof is complete.
\end{proof}

\begin{corollary}
Suppose that\ for some $p_{0}$, $0<p_{0}<\infty ,$ and for every $w_{0}\in 
\mathcal{A}_{\infty },$%
\begin{equation}
\sum_{k=1}^{\infty }w_{0}(k)f^{p_{0}}(k)\leq C\sum_{k=1}^{\infty
}w_{0}(k)g^{p_{0}}(k)\text{.}  \label{ac1}
\end{equation}%
Then for all $p,$ $0<p<\infty $ and all\ $w\in \mathcal{A}_{\infty },$%
\begin{equation}
\sum_{k=1}^{\infty }w(k)f^{p}(k)\leq C\sum_{k=1}^{\infty }w(k)g^{p}(k).
\label{acv1}
\end{equation}
\end{corollary}

\begin{proof}
Fix $r,$ $1<r<\infty $ and define a new family $\Omega _{r}$ consisting of
the pairs $(f^{p_{0}/r},g^{p_{0}/r})$, with $(f,g)\in \Omega .$ Then by
inequality (\ref{ac1})$,$ for for every $w_{0}\in \mathcal{A}_{r}\subset 
\mathcal{A}_{\infty }$ and $(f^{p_{0}/r},g^{p_{0}/r})\in \Omega _{r}$%
\begin{eqnarray*}
\sum_{k=1}^{\infty }w_{0}(k)\left( f^{p_{0}/r}(k)\right) ^{r}
&=&\sum_{k=1}^{\infty }w_{0}(k)\left( f^{p_{0}}(k)\right) \leq
C\sum_{k=1}^{\infty }w_{0}(k)g^{p_{0}}(k) \\
&=&C\sum_{k=1}^{\infty }w_{0}(k)\left( g^{p_{0}/r}(k)\right) ^{r}.
\end{eqnarray*}%
This gives the initial estimate (\ref{3,2}) for $\Omega _{r}$ and with $r$
in place of $p_{0}.$ Therefore by Theorem \ref{T3}, we have that for every $%
s,$ $1<s<\infty $ and $w\in \mathcal{A}_{s}$%
\begin{equation}
\sum_{k=1}^{\infty }w(k)\left( f^{p_{0}/r}(k)\right) ^{s}\leq
C\sum_{k=1}^{\infty }w(k)\left( g^{p_{0}/r}(k)\right) ^{s}.  \label{ac3}
\end{equation}%
Fix $p>1$ and $w\in \mathcal{A}_{\infty }.$ By Theorem \ref{Muck2}, since $%
\mathcal{A}_{p}$ are nested, we may assume without loss of generality that $%
w\in \mathcal{A}_{s}\subset \mathcal{A}_{\infty }$ for $s>p/p_{0}.$
Therefore, we can choose $r>1$ such that $\left( p_{0}/r\right) s=p$ and
then the inequality (\ref{ac3}) yields the inequality (\ref{acv1}). The
proof is complete.
\end{proof}

\end{document}